\documentclass[12pt]{amsart}
\usepackage[margin=1.0in]{geometry} 
\usepackage{esint,amsthm,amsmath,amssymb,amsfonts,graphicx,color,comment,enumerate,psfrag,caption,bbm,bm,hyperref,enumitem,mathrsfs}
\usepackage{mathtools}

\usepackage{tikz,pgf}
\usepackage{pgfplots}
\usetikzlibrary{calc}
\usetikzlibrary{patterns}

\allowdisplaybreaks

\DeclareMathOperator\supp{supp}

\newtheorem{lemma}{Lemma}[section]
\newtheorem{remark}{Remark}[section]
\numberwithin{equation}{section}
\newtheorem{theorem}{Theorem}[section]
\newtheorem{proposition}[theorem]{Proposition}

\newtheorem{corollary}[theorem]{Corollary}

\begin{document}
\title[Bochner-Riesz commutators on M\'etivier groups]{Bochner-Riesz commutators on M\'etivier groups: boundedness and compactness}

\author[Md N. Molla, J. Singh]
{Md Nurul Molla \and Joydwip Singh} 

\address[Md N. Molla]{Department of Mathematics and Statistics, Indian Institute of Science Education and Research Kolkata, Mohanpur--741246, West Bengal, India; Department of General Sciences, BITS Pilani Dubai Campus, International Academic City, Dubai, 345055, UAE}
\email{nurul.pdf@iiserkol.ac.in \& nurul@dubai.bits-pilani.ac.in}

\address[J. Singh]{Department of Mathematics and Statistics, Indian Institute of Science Education and Research Kolkata, Mohanpur--741246, West Bengal, India.}
\email{js20rs078@iiserkol.ac.in}



\begin{abstract} 
In this paper, we prove the boundedness and compactness properties of the Bochner-Riesz commutator associated with the sub-Laplacians on M\'etivier groups. We show that the smoothness parameter can be expressed in terms of the topological dimension, rather than the homogeneous dimension of the M\'etivier groups. One of the main novel aspects of our proof is the use of \emph{weighted Plancherel estimates} and \emph{truncated restriction-type estimates} to obtain the threshold in terms of the topological dimension.
\end{abstract}

\maketitle

\section{Introduction and statement of main results}
Let $G$ be a two-step stratified Lie group with $\mathfrak{g}$ as its Lie algebra, that is to say that $G$ is connected, simply connected nilpotent Lie group and the Lie algebra $\mathfrak{g}$ endowed with a vector space decomposition $\mathfrak{g} = \mathfrak{g}_1\oplus \mathfrak{g}_2$ into two non-trivial subspaces $\mathfrak{g}_1, \mathfrak{g}_2 \subseteq \mathfrak{g}$ such that $[\mathfrak{g}_1, \mathfrak{g}_1]$ = $\mathfrak{g}_2$ and $\mathfrak{g}_2$ is the center of $\mathfrak{g}$. We call $\mathfrak{g}_1$, $\mathfrak{g}_2$ as first and second layer respectively and let $\text{dim}\ \mathfrak{g}_1 =d_1$, $\text{dim}\ \mathfrak{g}_2 = d_2$. Let us assume that $X_1, \ldots , X_{d_1}$ be a basis of the first layer $\mathfrak{g}_1$ and $T_1, \ldots, T_{d_2}$ be a basis of the second layer $\mathfrak{g}_2$.  We further assume that we have an inner product $\langle \cdot\,, \cdot\rangle$ on $\mathfrak{g}$ such that $X_1, \ldots , X_{d_1}, T_1, \ldots, T_{d_2} $ is an orthonormal basis of $\mathfrak{g}$. The inner product $\langle \cdot\,, \cdot\rangle$ induces a norm on the dual space $\mathfrak{g}_2 ^{*}$ of $\mathfrak{g}_2$, which we denote by $|\cdot|$. Then any element $\mu \in\mathfrak{g}_2^{*}\setminus\{0\}$ gives rise to a skew-symmetric bilinear form $ \omega_{\mu}(x, x')= \mu([x,x'])$ for $ x, x' \in \mathfrak{g}_1$. Let $J_{\mu}$ be the skew-symmetric endomorphism such that
\begin{align*}
    \omega_{\mu}(x, x')=  \langle J_{\mu}x, x' \rangle , \quad x, x' \in \mathfrak{g}_1. 
\end{align*}
We then say $G$ is a M\'etivier group if $J_{\mu}$ is invertible for all $\mu \in\mathfrak{g}_2 ^{*}\setminus\{0\}$, and a Heisenberg type group if $J_{\mu} ^2 = - |\mu|^2\ \text{id}_{\mathfrak{g}_1}$ for all $\mu \in \mathfrak{g}_2 ^{*}\setminus \{0\}$. In the latter case, the group will be denoted by $\mathbb{H}$. In what follows, $G$ will denote the M\'etivier groups unless or otherwise specified. Note that the class of M\'etivier group contains Heisenberg type group, in fact the containment is strict, see \cite{Muller_Seeger_Singular_Spherical_maximal_operator_Nilpotent_2004} for more details. 
 
Corresponding to the basis $X_1, \ldots, X_{d_1}$ 
of the first layer $\mathfrak{g}_1$, we consider the associated sub-Laplacian $\mathcal{L}$ on $G$ which is defined by
\begin{align*}
     \mathcal{L}:= -(X_1 ^2 + \cdots + X_{d_1} ^2).
\end{align*}
The sub-Laplacian $\mathcal{L}$ is positive and (essentially) self-adjoint on $L^2(G)$. 
 Hence,  via functional calculus,  given a Borel measurable function $F: \mathbb{R} \rightarrow \mathbb{C}$ one can define spectral multiplier operator
 \begin{align*}
     F(\mathcal{L})= \int_{0} ^{\infty} F(\lambda) \, dE(\lambda),
 \end{align*}
where $E$ denotes the spectral measure of $\mathcal{L}$. Then from spectral theory, $F(\mathcal{L})$ is a bounded operator on $L^2(G)$ if and only if $F$ is $E$-essentially bounded. However, for $p \neq 2$, it is well known that the boundedness condition on $F$ alone is not enough to get the $L^p$-boundedness of $F(\mathcal{L})$, one requires more regularity on the multiplier $F$. In the context of Euclidean space $\mathbb{R}^n$, one such sufficient condition is given by celebrated Mihlin-H\"ormander multiplier theorem, which implies the $L^p$-boundedness of the spectral multiplier $F(-\Delta)$, which is basically radial Fourier multiplier. It states that $F(-\Delta)$ is bounded on $L^p(\mathbb{R}^n)$ for $1<p<\infty$, whenever
\begin{align}
\label{Mihlin-Hormander type condition}
    \|F\|_{L^2_{s,\text{sloc}}}:= \sup_{t>0} \|F(t \cdot)\eta\|_{L_s^2(\mathbb{R})} < \infty
\end{align}
for some $s>n/2$, where $\eta$ is some smooth function supported in $(0, \infty)$. In fact the condition $s>n/2$ is sharp (see \cite{Sikora_Wright_Imaginary_Power_Laplacian_2001}). After that, Mihlin-H\"ormander multiplier theorem has been generalized in many different settings. For instance, in the setting of stratified Lie group $\mathcal{G}$, such results were independently proved by Mauceri and Meda \cite{Mauceri_Meda_Multipliers_Stritified_group_1990} and by Christ \cite{Christ_Spectral_multiplier_Nilpotent_1991}. In particular, they showed that for the sub-Laplacian $L$ on $\mathcal{G}$, the operator $F(L)$ extends to a bounded operator on $L^p(\mathcal{G})$ for $p\in (1, \infty)$ whenever $F$ satisfies (\ref{Mihlin-Hormander type condition}) for some $s>Q/2$, where $Q$ is the homogeneous dimension of $\mathcal{G}$. For all homogeneous sub-Laplacian on the Heisenberg group $\mathbb{H}$, M\"uller and Stein \cite{Muller_Stein_Spectral_Multiplier-Heisenberg_1994} proved that the smoothness condition $s>Q/2$ can be replaced by $s> d/2$, where $d=d_1+d_2$ being the topological dimension of $\mathbb{H}$. In case of Heisenberg type group, similar result is independently obtained by Hebisch \cite{Hebisch_Spectral_Multiplier_Heisenberg_1993}.  Moreover, the threshold $s> d/2$ for $\mathbb{H}$ turned out to be sharp in the sense that it can not  be improved further. This discovery resulted in a surge of attention among  many researchers and sparked numerous current research in the theory of spectral multiplier, see \cite{Martini_Sikora_Grushin_Weighted_Plancherel_2012}, \cite{Martini_Muller_Multiplier_Grushin_2014}, \cite{Martini_Spectral_Multiplier_Heisenberg_Reiter_2015}, \cite{Marini_Multiplier_polynomial_Growth_2012}, \cite{Ahrens_Cowling_Martini_Muller_Quaternionic_Sphere_2020}, \cite{Casarino_Ciatti_Martini_Grushin_Sphere_2019}, \cite{Cowling_Klima_Sikora_Kohn_Laplacian_On_Sphere_2011} and references therein. Let us remark that, however in general for any 2-step Stratified Lie group, it still remains an open question whether the homogeneous dimension $Q$ in smoothness condition can be replaced by topological dimension $d$. Nevertheless, in case of M\'etivier group $G$,  Martini \cite{Marini_Multiplier_polynomial_Growth_2012} proved that the threshold can be pushed down again from $Q/2$ to $d/2$. In particular, this result also implies that Bochner-Riesz means associated to $\mathcal{L}$,
\begin{align*}
    S^{\alpha}(\mathcal{L}) &:= (1-t\mathcal{L})_+^{\alpha}, \quad t>0
\end{align*}
is bounded on $L^p(G)$ for all $1\leq p \leq \infty$ whenever $\alpha>(d-1)/2$, and the result is sharp. Therefore if one wants to study the boundedness of $S^{\alpha}(\mathcal{L})$ for $0<\alpha \leq (d-1)/2$, one should not expect boundedness to hold for all $p \in [1, \infty]$. Towards this direction very recently, Chen and Ouhabaz \cite{Chen_Ouhabaz_Bochner-Riesz_Grushin_2016} and Niedorf \cite{Niedorf_Bochner_Riesz_Grushin_2022} in the case of Grushin operators, Niedorf in Heisenberg type groups \cite{Niedorf_Spectral_Multiplier_Heisenber_Group_2024} and M\'etivier groups \cite{Niedorf_Metivier_group_2023} obtained the $p$-specific type Bochner-Riesz multiplier results. In fact for M\'etivier groups, Niedorf proved the following theorem.

For any $n \in \mathbb{N} \setminus \{0\}$, we define the Stein-Tomas exponent by $p_n=\frac{2(n+1)}{n+3}$. Let $p_{d_1, d_2} = p_{d_2}$ for $(d_1, d_2) \notin \{(8,6), (8,7) \}$, $p_{8,6} = 17/12$ and $p_{8,7} = 14/11$.

\begin{theorem}\cite[Theorem 1.2]{Niedorf_Metivier_group_2023}
\label{Theorem: Niedorf_Metivier_group}
    Let $1\leq p \leq p_{d_1, d_2}$. Suppose that $\alpha > d(1/p -1/2)- 1/2$ with $d=d_1+d_2$. Then the Bochner-Riesz means $S^{\alpha}(\mathcal{L})$ is bounded on $L^p(G)$ uniformly in $t\geq 0$.
\end{theorem}

In this paper, we study boundedness and compactness properties of the Bochner-Riesz commutator generated by $ S^{\alpha}(\mathcal{L})$ and $b \in BMO^{\varrho}(G)$, the space of functions of bounded mean oscillation on $G$ (for definition see section \ref{sec: preli}). For $f \in \mathcal{S}(G)$, the Bochner-Riesz commutator associated to the sub-Laplacian $\mathcal{L}$ is denoted by $[b, S^{\alpha}(\mathcal{L})] $ and is defined by 
\begin{align*}
    [b, S^{\alpha}(\mathcal{L})]f = b\, S^{\alpha}(\mathcal{L})f - S^{\alpha}(\mathcal{L})(bf) .
\end{align*} 

Before presenting our results, it would be convenient to revisit some relevant results from the literature. In the classical setting, the Bochner-Riesz commutator was studied by many mathematicians, see for example \cite{Hu_Lu_Bochner_Riesz_Commutator_1996}, \cite{Lu_Yan_Bochner_Riesz_Book_2013} and \cite{Lu_Xia_Bochner_Riesz_Commutator_2007}. In a very recent article \cite{Chen_Tian_Ward_Commutator_Bochner_Riesz_Elliptic_2021}, the authors obtained $L^q$-boundedness of commutators $[b, S^{\alpha}_R (L)] $ of a BMO function $b$
and the Bochner-Riesz means $S^{\alpha}_R (L)$ for a class of nonnegative self-adjoint elliptic operators $L$ on doubling metric measure space $(X, \mu)$ with  ball volume being polynomial type. To be precise, assuming that $L$ satisfies the finite speed of propagation property and the crucial  spectral measure estimate: 
 \begin{align}\label{spectral measure estimate}
    \|dE_{\sqrt{L}} (\lambda)  \|_{L^p \rightarrow L^{p'}} \leq C \lambda^{n(1/p - 1/p') -1}, \quad \lambda >0, 
\end{align}
for some $1 \leq p <2$, they obtained $L^q$-boundedness of $[b, S^{\alpha}_R (L)] $ for all $p< q< p'$ with the smooth condition $\alpha> n(1/p-1/2) -1/2$, $n$ being the homogeneous dimension of $X$. It is known that for metric measure spaces if the ball volume is polynomial type, then (\ref{spectral measure estimate}) is equivalent to the Stein-Tomas restriction condition:  for any $R > 0$ and all Borel functions $F$ such that $\supp F \subseteq [0,R]$, 
\begin{align}\label{restriction estimate}
\| F(\sqrt{L})\chi_{B(x, r)} \|_{L^p \rightarrow L^{2}} \leq C R^{n(\frac{1}{p} - \frac{1}{2})} \|\delta_R F\|_{L^2},
\end{align}
for all $x\in X$ and $r \geq 1/R$. For further details, see \cite[Proposition I.4]{Chen_Ouhabaz_Sikora_Yan_Restriction_Estimate_Bochner-Riesz_2016}.

In the present context of M\'etivier groups $G$, although the ball volume is polynomial type and the sub-Laplacian $\mathcal{L}$ satisfies the finite speed of propagation property, to the best of our knowledge, it is still remains open question whether one can establish (\ref{restriction estimate}) type estimates for $\mathcal{L}$. A lot of investigation has been made in this direction, for instance in case of Heisenberg group, the range of $p$ for which the restriction estimate hold is $p=1$, this is because of the one-dimensional center of the Heisenberg group \cite{Muller_restriction_Heisenberg_group_1990}. For Heisenberg type groups, dimension of the center is bigger than $1$, there restriction type estimate was known to be true, see \cite{Liu_Wang_Restriction_Heisenberg_Type_2011}, \cite{Thangavelu_Restriction_product_Heisenberg_group_1991}. But outside the class of Heisenberg type groups, unfortunately restriction estimate is still unknown, see \cite[Reviewer's remark]{Liu_Zhang_Metivier_joint_functional_2018}, also see remark of \cite{Niedorf_Spectral_Multiplier_Heisenber_Group_2024} on \cite{Casarino_Ciatti_Restriction_Metivier_groups_2013}. Therefore, it is a natural question whether $L^p$-boundedness of $[b, S^{\alpha}(\mathcal{L})]$ still holds with smoothness parameter $\alpha> Q(1/p-1/2) -1/2$, where $Q$ being  the homogeneous dimension of $G$. Because of the sub-Riemannian geometry of the underlying group, one can also ask whether smoothness parameter can be replaced by a weaker condition $\alpha >d(1/p- 1/2)- 1/2$, with $d$ as the topological dimension of $G$. We answer this question positively, using a weaker version of restriction type estimate proved in \cite{Niedorf_Restriction_Stratified_group_2023}.

For all $d_1, d_2 \geq 1$, let $p_{d_1, d_2}$ be the same as just before the Theorem \eqref{Theorem: Niedorf_Metivier_group} and $L^2_{\beta}(\mathbb{R})$ denote the $L^2$ Sobolev space of order $\beta \geq 0$. Then following is our first main result of this paper.
\begin{theorem}
\label{Theorem: Multiplier for Commutator} 
Let $1\leq p \leq  p_{d_1, d_2}$ and $\beta>d(1/p-1/2)$. Then for any even Borel function $F$ with $\supp{F} \subseteq [-1,1]$ and $F \in L^2_{\beta}(\mathbb{R})$, we have
\begin{align*}
    \|[b, F(t \sqrt{\mathcal{L}})]f\|_{L^q(G)} &\leq C \|b\|_{BMO^{\varrho}(G)} \|F\|_{L^2_{\beta}(\mathbb{R})} \|f\|_{L^q(G)},
\end{align*}
    for all $p<q<p'$ and uniformly in $t \in (0, \infty)$.
\end{theorem}
The following result is a consequence of the above theorem, which gives the boundedness of Bochner-Riesz commutator $[b, S^{\alpha}(\mathcal{L})]$.
\begin{theorem}
\label{Theorem: Bochner-Riesz commutator on Metivier}
    If $1\leq p \leq  p_{d_1, d_2}$ and $\alpha>d(1/p-1/2) - 1/2$,  then the Bochner-Riesz commutator $[b, S^{\alpha}(\mathcal{L})]$ is bounded on $L^q(G)$ whenever $p< q< p'$.
\end{theorem}

Since Heisenberg type groups falls into the class of M\'etivier groups, in particular, the above result is also applicable there. Notice that in the above theorems the range of $p$ is in between $1$ and $p_{d_2}$, for all $d_1, d_2 \geq 1$ except the two points $(8,6)$ and $(8,7)$. For $(d_1, d_2) \in \{(8,6), (8,7) \}$, $p_{6}=14/9>17/12$ and $p_{7}=8/5>14/11$. However, if one restricts attention to the sub-Laplacian on the Heisenberg type groups, then we can improve the above result for all $d_1, d_2 \geq 1$, with the range of $p$ is $1\leq p \leq p_{d_2}$. In fact, we have the following result.

\begin{theorem}
\label{Theorem: Bochner-riesz commuator on Heisenber-type group}
    Let $\mathbb{H}$ be Heisenberg type group and $\mathcal{L}$ be its associated sub-Laplacian. If $1\leq p \leq  2(d_2 + 1)/(d_2 + 3)$ and $\alpha>d(1/p-1/2) - 1/2$,  then the Bochner-Riesz commutator $[b, S^{\alpha}(\mathcal{L})]$ is bounded on $L^q(\mathbb{H})$ whenever $p< q< p'$.
\end{theorem}

Our next discussion will be regarding the compactness property of the commutators. In \cite{Uchiyama_compactness_commutator_1978}, Uchiyama proved that the commutator  of Riesz transform on $\mathbb{R}^n$ is compact on $L^p(\mathbb{R}^n)$ with $1<p< \infty$ if and only if the symbol belongs to $CMO(\mathbb{R}^n)$, the closure of $C_c^{\infty}(\mathbb{R}^n)$ functions in the $BMO(\mathbb{R}^n)$ norm. In recent times, the compactness property of commutators for many different class of operators in various settings has been studied widely, such as the Riesz transform associated with the sub-Laplacian on stratified Lie groups, the Riesz transform associated with Bessel operator, Bochner-Riesz commutator on $\mathbb{R}^n$, the Cauchy’s integrals on $\mathbb{R}$, and the Calder\'on-Zygmund operator with homogeneous kernel to name a few. Interested readers are referred to \cite{Chen_Duong_Li_Wu_Compactness_Riesz_Transform_Stratified_group_2019}, \cite{Duong_Li_Mo_Wu_Yang_compactness_Bessel_2018}, \cite{Bu_Chen_Hu_Bochner_Riesz_commutator_2017}, \cite{Chaffee_Chen_Han_Torres_Ward_Compact_bilinear_2018}, \cite{Tao_Yang_Yuan_Zhang_compact_ball_Banach_2023} and reference therein for further details. Another aim of this paper is to study the compactness property of $[b, S^{\alpha}(\mathcal{L})]$ in the setting of M\'etivier groups $G$. Let $CMO^{\varrho}(G)$ be the closure of $C^{\infty} _c (G)$ under the norm of $BMO^{\varrho}(G)$. Then the following is the second main result of this paper.
\begin{theorem}
\label{Theorem: Compactness of Bochner-Riesz commutator}
    Let $1\leq p \leq  p_{d_1, d_2}$ and $\alpha>d(1/p-1/2)-1/2$. If $b \in CMO^{\varrho}(G)$, then the Bochner-Riesz commutator $[b, S^{\alpha}(\mathcal{L})]$ is a compact operator on $L^q(G)$ for all $p<q<p'$.
\end{theorem}

\begin{remark}
    It is worth mentioning that all the above results, Theorem \ref{Theorem: Multiplier for Commutator}, Theorem \ref{Theorem: Bochner-Riesz commutator on Metivier}, Theorem \ref{Theorem: Bochner-riesz commuator on Heisenber-type group} and Theorem \ref{Theorem: Compactness of Bochner-Riesz commutator} are stated with the smoothness parameter $\alpha$ in terms of the topological dimension, $d$ of the underlying group $G$. The smoothness parameter what we get here is $\alpha>d(1/p-1/2)-1/2$, which is same as the Bochner-Riesz means appeared in \cite[Theorem 1.1]{Niedorf_Metivier_group_2023} and \cite[Theorem 1.1]{Niedorf_Spectral_Multiplier_Heisenber_Group_2024}, and which are sharp in view of \cite{Martini_Muller_Golo_Spectral_Multiplier_Lower_Regularity_2023}. One can ask the similar questions for other sub-Laplacians. We are also working on Grushin operator, but this result will appear in somewhere else \cite{Molla_Singh_Bochner_Riesz_commutator_Grushin}. 
\end{remark}

\emph{An outline of our strategy for proving Theorem \ref{Theorem: Multiplier for Commutator}.}
As mentioned before, one of the major difficulties arising in our setting is the absence of spectral measure estimate (\ref{spectral measure estimate}) for $\mathcal{L}$. We would like to mention that the arguments given in \cite{Chen_Tian_Ward_Commutator_Bochner_Riesz_Elliptic_2021}  have a limited scope as even if one has (\ref{spectral measure estimate}) and follows  \cite{Chen_Tian_Ward_Commutator_Bochner_Riesz_Elliptic_2021}, yet one can merely obtain $L^q$-boundedness of $[b, S^{\alpha}(\mathcal{L})]$ with the smoothness condition being expressed in terms of the homogeneous dimension of $G$. One of the main tools that has been used there to reduce the commutator boundedness to the boundedness of the corresponding operator via \cite[Lemma 3.2]{Chen_Tian_Ward_Commutator_Bochner_Riesz_Elliptic_2021}. We can not use this lemma as a black box, as this will lead to the smoothness parameter in terms of $Q$. At this point, one may desire to use weighted restriction estimate for $\mathcal{L}$ but unfortunately, we still lack of such tools because such estimates does not hold here, see \cite[Section 8]{Niedorf_Spectral_Multiplier_Heisenber_Group_2024}. Thus a refinement of arguments in \cite{Chen_Tian_Ward_Commutator_Bochner_Riesz_Elliptic_2021} is necessary. Next, to circumvent the absence of (\ref{spectral measure estimate}), we make use of a weaker version of restriction type estimates,  Proposition \ref{Theorem: Truncated Restriction Estimate strong form}, which is due to \cite{Niedorf_Metivier_group_2023},  where the operator $F(\mathcal{L})$ is further truncated dyadically  along the spectrum of $T:= ( - (T_1^2 + \cdots + T_{d_2} ^2))^{1/2}$. In this approach, we also need to take care of a discrete norm of $F$ appearing in Proposition (\ref{Theorem: Truncated Restriction Estimate strong form}). Nevertheless, such an approach is limited to dealing with the error part of this additional truncation (see \eqref{Inequality: Error part for truncation}). In order to address the main part, the sub-Riemannian geometry of the underlying manifold is taken into account, in particular we use (\ref{Decomposition of ball into Euclidean balls}) crucially. This helps us to further split the main part into two summands (see \eqref{Decomposition of the operator in error and main part}). The negligible part of the summand is tackled through the weighted Plancherel estimates (Proposition \ref{Theorem : First layer weighted Plancherel}), whereas the significant part of the summand is taken care of by Proposition \ref{Theorem: Truncated Restriction Estimate strong form} and Proposition \ref{prop: Radon-Hurwitz number}, which is about the numerology of the M\'etivier groups. Because of Proposition \ref{prop: Radon-Hurwitz number}, we have to remove three points for the Stein-Tomas range. Nevertheless one can recover the full range of $p$, by carefully analyzing the point $(d_1, d_2)=(4,3)$ as done in \cite[Section 8]{Niedorf_Metivier_group_2023}.
 
\emph{An outline of our strategy for proving Theorem \ref{Theorem: Compactness of Bochner-Riesz commutator}.} A careful reading of the proof of the corresponding result in euclidean spaces \cite{Bu_Chen_Hu_Bochner_Riesz_commutator_2017} reveals that the proof relies on the  three main ingredients namely, Fourier transform estimates, approximation to the identity, and some refined estimates obtained by C. Fefferman in \cite{Fefferman_Spherical_multiplier_1973}. We would like to point out  that such arguments are not very effective in the setting of M\'etivier group $G$. A significant aspect of our proof is the use of the fact that the norm limit of compact operators is again a compact operator. First we decompose the operator in Fourier transform side and on each pieces we employ the Kolmogorov-Riesz compactness type theorem Proposition \ref{Proposition: Kolmogorov-Riesz compactness theorem}. In order to do so we need good pointwise control over the kernel and the gradient of the kernel. These estimates we have proved in Proposition \ref{Proposition: Weighted Plancherel with L infinity condition}, which helps us to get our required result in this settings as well which could be an independent interest.

The plan of the paper is as follows. In the next  Section (\ref{sec: preli}) we gather some well known results concerning the sub-Riemannian geometry of $G$, the space of bounded mean oscillation, spectral decomposition of $-J_{\mu}^2$, an explicit formula for convolution kernel, integration of a norm function and an useful result. Section \ref{section: restriction and kernel estimate} is devoted to developing some key devices related to truncated restriction estimates, weighted Plancherel estimates and pointwise weighted kernel estimates that we need for further study. Finally, in Section \ref{Section: proof of boundedness of commutator} and Section \ref{section: proof of copactness} we present the proof of main results of this paper.

Let $\mathbb{N}=\{0,1,2,\ldots\}$. For a Lebesgue measurable subset $E$ of $\mathbb{R}^d$, we denote by $\chi_E$ the characteristic function of the set $E$. We use letter $C$ to indicate a  positive constant independent of the main parameters, but may vary from line to line. We shall use the notation $f \lesssim g$ to indicate $f \leq Cg$ for some $C > 0$, and whenever $f \lesssim g\lesssim f$, we shall write $f \sim g$. Also, we write $f \lesssim_{\epsilon} g$ when the implicit constant $C$ may depend on a parameter like $\epsilon$.  Whenever $A>0$ is very small we frequently replace the quantity $2^A$ by $2^{\epsilon_1}$ for some $\epsilon_1 >0$ very small. Also, for any ball $ B := B(x, r)$ with centered at $x$ and radius $r$, the notation $\kappa B=B(x, \kappa r)$ stands for the concentric dilation of $B$ by $\kappa >0$. Moreover, for a measurable function $f$, we denote the average of $f$ over $B$ as  $f_{B} = \frac{1}{|B|} \int_{B} f(x) \, dx$. For any function $G$ on $\mathbb{R}$, we define $\delta_R G(\eta) = G(R \eta) $ for any $R>0$. For $f,g \in \mathcal{S}(G)$, let $f * g$ denotes their group convolution given by
\begin{align*}
    (f*g)(x,u) &= \int_G f(x',u') g((x',u')^{-1}(x,u)) \ d(x',u'), \quad \quad (x,u) \in G.
\end{align*}

\section{Preliminaries for the M\'etivier groups}\label{sec: preli}
Let $G$ be the M\'etivier group and $\mathfrak{g} = \mathfrak{g}_1 \oplus \mathfrak{g}_2$ be the stratification of its Lie algebra $\mathfrak{g}$. Let $X_1, \ldots, X_{d_1}$ and $T_1, \ldots, T_{d_2}$ be a basis of first layer $\mathfrak{g}_1$ and of the second layer $\mathfrak{g}_2$ respectively. By means of the global diffeomorphism $\exp: \mathfrak{g} \rightarrow G$, the group $G$ is identified with its Lie algebra $\mathfrak{g}$, which can be again identified with $\mathbb{R}^{d_1} \times \mathbb{R}^{d_2}$. Thus, we write an element $g$  of $G$  as $(x, u)$, where $x\in \mathbb{R}^{d_1}$ and $u \in \mathbb{R}^{d_2}$.  Moreover, the group product formula on $G$ is determined by Baker-Campbell-Hausdorff formula,
 \begin{align}
     (x, u)(x', u')= (x+x', u +u'+ \tfrac{1}{2} [x,x']). 
 \end{align}
And the identity element of $G$ is $(0,0)$, while the inverse $(x,u)^{-1}$ of an element $(x,u)$ is given by $(x,u)^{-1}= (-x, -u)$. The Haar measure $d(x,u)$ on $G$ corresponds to the Lebesgue measure on $\mathfrak{g}$, via the exponential map. There is a natural family of dilations $  \{ \delta_{R} : R >0\}$ on $G$ which is given by
 \begin{align}\label{Dilation on metivier}
     \delta_R (x, u)= (Rx, R^2u),  \quad (x,u) \in G.
 \end{align}
With respect to this dilation the following 
\begin{align}
\label{Def: homogeneous norm}
    \|(x,u)\| := (|x|^4 + |u|^2)^{1/4}, \quad (x,u) \in G,
\end{align}
defines a homogeneous norm on $G$.

Let $\varrho$ denote the Carnot-Carath\'eodory metric associated to the vector fields $X_1, \ldots, X_{d_1}$. In view of Chow–Rashevskii theorem, $\varrho$ induces the Euclidean topology of $G$. By the property that any two homogeneous norms on a homogeneous Lie group are equivalent, we have 
\begin{align*}
     \varrho(g,h) \sim \|g^{-1}h\|, \quad \text{for all}\ g,h \in G.
\end{align*}
Moreover, since these vector fields are left-invariant, $\varrho$ is left-invariant in the sense that  
\begin{align*}
    \varrho(ag, ah) &= \varrho(g,h) \quad \text{for all} \ a, g, h \in G.
\end{align*} 
 
Let $B^{\varrho} ((x,u),R)$ and $\Bar{B}^{\varrho} ((x,u),R)$ denote respectively the open and closed ball of radius $R>0$ centered at $(x,u) \in G$ with respect to $\varrho$. In the sequel, we omit the superscript $\varrho$ and simply write $B$ or $B((x,u),R)$ to indicate an open ball with respect to the Carnot-Carath\'eodory distance. The volume of the ball  is given by
\begin{align*}
    |B((x,u),R)| \sim R^Q |B(0,1)|,
\end{align*}
where $|\cdot|$ denote the Lebesgue measure and $Q=d_1+2d_2$. Thus, the metric measure space $(G, \varrho, |\cdot|)$ is indeed a space of homogeneous type, with homogeneous dimension $Q$. We also call $d=d_1+d_2$ to be the topological dimension of $G$.

From \eqref{Def: homogeneous norm}, one can easily see that, there exists a constant $C>0$ such that
\begin{align}
\label{Decomposition of ball into Euclidean balls}
    B(0, R) \subseteq B^{|\cdot|}(0, C R) \times B^{|\cdot|}(0, C R^2) \subseteq \mathbb{R}^{d_1} \times \mathbb{R}^{d_2} ,
\end{align}
where $B^{|\cdot|}(a, R)$ denotes the ball of radius $R$ and centered at $a$ with respect to Euclidean distance.

Note that on M\'etivier groups we always have $d_1>d_2$. This is true because the map $\mu \to \omega_{\mu}(\cdot, x')$ from $\mathfrak{g}_2^{*} \to (\mathfrak{g}_1/\mathbb{R}x')^{*}$ is injective for $x' \neq 0$. The next proposition tells us that for M\'etivier groups $d_1$, the dimension of the first layer is much larger than $d_2$, the dimension of the second layer, except for few cases.
\begin{proposition}\cite[Proposition 7.1]{Niedorf_Metivier_group_2023}
\label{prop: Radon-Hurwitz number}
Let $G$ be a  M\'etivier group of topological dimension $d=d_1+d_2$ with center of dimension $d_2$. Then for $(d_1, d_2) \notin \{ (4, 3), (8,6), (8, 7) \}$, we have $d_1 > 3d_2/2$.
\end{proposition}  

The theory of $BMO$ spaces on space of homogeneous type is well known. The function space of bounded mean oscillation on $G$ is denoted by $BMO^{\varrho}(G)$ and is defined by 
\begin{align}
    BMO^{\varrho}(G) =\{f \in L^1_{\text{loc}}(G): \|f\|_{BMO^{\varrho}(G)} < \infty \},
\end{align}
  where
\begin{align}
    \|f\|_{BMO^{\varrho}(G)}:= \sup_{B}\left(\frac{1}{|B|}\int_{B}|f(y,t)-f_{B}| \, d(y,t) \right) .
\end{align}

The celebrated John–Nirenberg inequality says that any function in $BMO^{\varrho}(G)$ has exponential decay of its distribution function. An important consequence of this distribution inequality is the $L^q$ characterization of $BMO^{\varrho}(G)$ norms:
\begin{align}
\label{Cherecterisation of BMO}
    \|f\|_{BMO^{\varrho}(G)} \sim \sup_{B}\left(\frac{1}{|B|}\int_{B}|f(y,t)-f_B|^q \, d(y,t) \right)^{\frac{1}{q}},
\end{align}
for any $1<q<\infty$. More details can be found in \cite{Ding_lee_Lin_Hardy_BMO_General_Sets_2014}. A simple but useful property of $BMO^{\varrho}(G)$ function is the following. Let $f$ be in $BMO^{\varrho}(G)$. Given a ball $B$ and a positive integer $k$ we have
\begin{align}
\label{Difference of average in terms of BMO}
    |b_{B}- b_{2^k B}| &\leq 2^Q\, k\, \|f\|_{BMO^{\varrho}(G)}.
\end{align}

Now we will discuss about the kernel representation of $m(\mathcal{L})$, for some appropriate function $m$. Recall that corresponding to the basis $\{X_1, \ldots, X_{d_1}\}$ and $\{T_1, \ldots, T_{d_2}\}$ of the first layer $\mathfrak{g}_1$ and $\mathfrak{g}_2$ respectively, the sub-Laplacian $\mathcal{L}$ on $G$ is defined by $ \mathcal{L}= -(X_1^2 + \cdots + X_{d_1}^2)$. Then the operators $\mathcal{L}, -iT_1, -iT_2, \cdots, -i T_{d_2}$ are essentially self-adjoint on $L^2(G)$ and commutes strongly, therefore they admit a joint functional calculus. In particular, the operator $\mathcal{L}$ and $T:= -( T_1^2+ \cdots + T_{d_2}^2)^{1/2}$ admit a joint functional calculus. Therefore, for any bounded Borel function  $m: \mathbb{R} \times \mathbb{R} \to \mathbb{C}$, the operator $m(\mathcal{L}, T)$ is well defined and bounded on $L^2(G)$.  In addition, by imposing an appropriate condition on $m$, it can be shown that $m(\mathcal{L}, T)$ possesses a convolution kernel $\mathcal{K}_{m(\mathcal{L},T)}$ and moreover, an explicit formula for $\mathcal{K}_{m(\mathcal{L},T)}$ can be obtained in terms of rescaled Laguerre functions. Since $G$ is a M\'etivier group, the endomorphism $J_{\mu}$ is invertible for all $\mu \in \mathfrak{g}_2^*$. The next proposition gives us the spectral decomposition of $-J_{\mu}^2$ for $\mu$ lies in some Zariski open subset of $\mathfrak{g}_2^*$.
\begin{proposition}\cite[Proposition 3.1]{Niedorf_Metivier_group_2023}
\label{Prop: structural properties}
    There exists a non-empty, homogeneous Zariski-open subset $\mathfrak{g}_{2,r}^{*}$ of $\mathfrak{g}_2^{*}$, numbers $N \in \mathbb{N}\setminus \{0\}$, $\mathbf{r}=(r_1, \ldots, r_N) \in (\mathbb{N} \setminus \{0\})^N$, a function $\mu \to \mathbf{b}^{\mu}=(b_1^{\mu}, \ldots, b_N^{\mu}) \in [0, \infty)^{N}$ on $\mathfrak{g}_2^{*}$, functions $\mu \mapsto P_{n,{\mu}} $ on $\mathfrak{g}_{2,r}^{*}$ with $P_{n,{\mu}} : \mathfrak{g}_1 \to \mathfrak{g}_1$, $n \in \{1, \ldots, N\}$ and a function $\mu \mapsto R_{\mu} \in O(d_1)$ on $\mathfrak{g}_{2,r}^{*}$ such that
    \begin{align*}
        -J_{\mu}^2 &= \sum_{n=1}^N (b_n^{\mu})^2 P_{n,{\mu}} \quad \text{for all} \ \mu \in \mathfrak{g}_{2,r}^{*},
    \end{align*}
    with $P_{n,{\mu}} R_{\mu} = R_{\mu} P_n$, $J_{\mu}(\text{ran}\  P_{n,{\mu}}) \subseteq \text{ran}\  P_{n,{\mu}}$ for the range of $P_{n,{\mu}}$ for all $\mu \in \mathfrak{g}_{2,r}^{*}$, and all $n \in \{1, \ldots, N\}$, where $P_n$ denotes the projection from $\mathbb{R}^{d_1}= \mathbb{R}^{2 r_1}\oplus \cdots \oplus \mathbb{R}^{2 r_N}$ onto the $n$-th layer, where
    \begin{enumerate}
        \item the function $\mu \to b_n^{\mu}$ are homogeneous of degree $1$ and continuous on $\mathfrak{g}_2^{*}$, real analytic on $\mathfrak{g}_{2,r}^{*}$, and satisfy $b_n^{\mu}>0$ for all $\mu \in \mathfrak{g}_{2,r}^{*}$ and $n \in \{1, \ldots, N\}$, and $b_n^{\mu} \neq b_{n'}^{\mu}$ if $n \neq n'$ for all $\mu \in \mathfrak{g}_{2,r}^{*}$ and $n, n' \in \{1, \ldots, N\}$,
        \item the functions $\mu \to P_{n,{\mu}}$ are (component wise) real analytic on $\mathfrak{g}_{2,r}^{*}$, homogeneous of degree $0$, and the maps $P_{n,{\mu}}$ are orthogonal projections on $\mathfrak{g}_1$ of rank $2 r_n$ for all $\mu \in \mathfrak{g}_{2,r}^{*}$, with pairwise orthogonal ranges,
        \item $\mu \to R_{\mu}$ is a Borel measurable function on $\mathfrak{g}_{2,r}^{*}$ which is homogeneous of degree $0$ and there is a family $(U_{\ell})_{\ell \in \mathbb{N}}$ of disjoint Euclidean open subsets $U_{\ell} \subseteq \mathfrak{g}_{2,r}^{*}$ whose union is $\mathfrak{g}_{2,r}^{*}$, up to a set of measure zero such that $\mu \to R_{\mu}$ is (component wise) real analytic on each $U_{\ell}$.
   \end{enumerate}
\end{proposition}

For $k, r \in \mathbb{N}$ and $\lambda>0$, let $\varphi_{k}^{(\lambda, r)}$ denote the $\lambda$-rescaled Laguerre function defined by
\begin{align*}
    \varphi_{k}^{(\lambda, r)}(z) &= \lambda^r L^{r-1}_k(\tfrac{1}{2}\lambda |z|^2) e^{-\frac{1}{2}\lambda |z|^2}, \quad z \in \mathbb{R}^{2r},
\end{align*}
where $L^{r-1}_k$ is the $k$-th Laguerre polynomial of type $r-1$ (see \cite{Thangavelu_Lectures_Hermite_1993}).

With the same notation as in the above proposition, the following result contains an explicit formula for the kernel of $H(\mathcal{L}, T)$.
\begin{proposition}\cite[Proposition 3.10]{Niedorf_Metivier_group_2023}
\label{Prop: Kernel expression}
    If $m: \mathbb{R} \times \mathbb{R} \to \mathbb{C}$ is a Schwartz function, then $m(\mathcal{L}, T)$ possesses a convolution kernel $\mathcal{K}_{m(\mathcal{L}, T)} \in \mathcal{S}(G)$, that is,
    \begin{align*}
        m(\mathcal{L}, T) f &= f * \mathcal{K}_{m(\mathcal{L}, T)} \quad \text{for all} \ f \in \mathcal{S}(G) .
    \end{align*}
    Moreover, for $x \in \mathfrak{g}_1$ and $u \in \mathfrak{g}_2$,  we have
    \begin{align*}
        \mathcal{K}_{m(\mathcal{L}, T)}(x,u) &= \frac{1}{(2 \pi)^{d_2}} \int_{\mathfrak{g}_{2,r}^{*}} \sum_{\mathbf{k} \in \mathbb{N}^N} m(\lambda_{\mathbf{k}}^{\mu}, |\mu|) \left[\prod_{n=1}^N \varphi_{k_n}^{(b_n^{\mu}, r_n)}(R_{\mu}^{-1} P_{n,{\mu}}x) \right] e^{i \langle \mu, u \rangle} \ d \mu ,
    \end{align*}
    with $\lambda_{\mathbf{k}}^{\mu} = \sum_{n=1}^{N} (2 k_n + r_n)b_n^{\mu}$.
\end{proposition}

We sometimes need the integration of the homogeneous norm on outside of some ball. The proof of the following lemma is easy and follows from just breaking the integral in disjoint union of annular regions, therefore we omit it here.
\begin{lemma}
\label{lemma: outside distance}
    Let $R, r> 0$. Then for any $s> Q$ we have
    \begin{align*}
        \int_{\|(x, u)\| > r} \frac{d(x,u)}{\big( 1 + R\|(x, u)\| \big)^s} &\leq C R^{-s} r^{-s + Q}.
    \end{align*}
\end{lemma}

We conclude this section by stating the following useful lemma, that we need in the proof of our results which can be proved using similar argument from the proof of \cite[Lemma 3.2]{Chen_Tian_Ward_Commutator_Bochner_Riesz_Elliptic_2021}.

\begin{lemma}
\label{Deduction: commutator to operator} 
Let $U$ be a linear operator with kernel $\mathcal{K}_{U}$ such that $\supp \mathcal{K}_U \subseteq \mathcal{D}_{r} := \{ (x, y) \in G \times G : \varrho(x, y)\leq r \}$ for some $r>0$. Assume $b \in BMO^{\varrho}(G)$. Then there exists a sequence $\{(x_k, u_k)\}_{k \in \mathbb{N}}$ in $G$, disjoint sets $\widetilde{B}_k \subseteq B_k := B\bigl((x_k,u_k), r \bigr)$ such that $B((x_k,u_k), \tfrac{r}{20}) \cap B((x_j,u_j), \frac{r}{20}) = \emptyset$ for $j \neq k$ and for $1\leq q <\infty$ we have
\begin{align*}
   \|[b, U]f\|_{L^q(G)}^q & \leq C \sum_k \Bigl(\| \chi_{4B_k}\, (b - b_{B_0})\, U (\chi_{\widetilde{B}_k} f) \|_{L^q(G)}^q + \| \chi_{4 B_k}\, U(\chi_{\widetilde{B}_k}\, (b - b_{B_0}) f )\|_{L^q(G)}^q  \Bigr) ,
\end{align*}
where $B_0 = B(0, r)$.
\end{lemma}

\section{Restriction and Kernel estimates}
\label{section: restriction and kernel estimate}
In this section we discuss about the restriction type estimates (Proposition \ref{Theorem: Truncated Restriction Estimate strong form}, Corollary \ref{Inequality: Truncated restriction equation weak form I}), weighted Plancherel estimates with respect to the first layer (Proposition \ref{Theorem : First layer weighted Plancherel}) and pointwise weighted kernel estimates (Proposition \ref{Proposition: Weighted Plancherel with L infinity condition}) for the sub-Laplacians on M\'etivier groups. One of the main ingredients in proving the boundedness of Bochner-Riesz means in the Stein-Tomas range is the use of restriction type estimates, which goes back to the idea of Stein \cite{Fefferman_Strongly_Singular_integral_1970}. Recently in \cite{Niedorf_Restriction_Stratified_group_2023}, the author proved truncation restriction type estimate on any two step stratified Lie groups, but with the presence of discrete norm instead of $L^2$-norm, which we discuss below.

For any bounded Borel function $\sigma : \mathbb{R} \to \mathbb{C}$ such that $\supp \sigma \subseteq [0,8]$, we define the following discrete norm
\begin{align*}
    \|\sigma\|_{N, 2}= \Biggl ( \frac{1}{N} \sum_{K\in \mathbb{Z}} \sup_{\lambda \in [\frac{K-1}{N}, \frac{K}{N}]} |\sigma(\lambda)|^2   \Biggr)^{\frac{1}{2}}, \quad N>0.
\end{align*}
Regarding the above discrete norm, we have the following estimates:
\begin{align}
\label{discrete norm is dominated by sup norm}
    \|\sigma\|_{N,2} &\leq C\, \|\sigma\|_{L^{\infty}} ,
\end{align}
and for $s>1/2$,
\begin{align}
\label{equation: cowling-sikora norm relation}
 \|\sigma\|_{L^2}\leq  \|\sigma\|_{N,2} \leq C \left(\|\sigma\|_{L^2} + N^{-s} \|\sigma\|_{L^2_s} \right) ,
\end{align}
where the first estimate follows from the definition and for the second estimate, see \cite[eq. (1.7)]{Niedorf_Restriction_Stratified_group_2023}.

Let $\Theta : \mathbb{R} \to [0,1]$ be an even $C_c^{\infty}$ function supported in $[-2,-1/2] \cup [1/2, 2]$, such that 
\begin{align}
\label{Definition: Cutoff function theta}
    \sum_{M \in \mathbb{Z}} \Theta_{M}(\tau) =1 ,
\end{align}
where $\Theta_{M}(\tau) = \Theta(2^{M} \tau)$. Also, let $F : \mathbb{R} \to \mathbb{C}$ is a bounded Borel function supported in $ [1/8, 8]$. Then for $M \in \mathbb{Z}$, we define $F_M : \mathbb{R} \times \mathbb{R} \to \mathbb{C}$ by 
\begin{align}
\label{Introducing theta in multiplier}
    F_M(\kappa, \tau) := F(\sqrt{\kappa}) \Theta( 2^{M} \tau) \quad \text{for} \quad \kappa \geq 0,
\end{align}
and $F_M(\kappa, \tau) = 0$ else.

Then we have the following \emph{truncated} restriction type estimate.
\begin{proposition}\cite[Theorem 1.1]{Niedorf_Metivier_group_2023}
\label{Theorem: Truncated Restriction Estimate strong form}
Suppose $1 \leq p \leq \min{\{p_{d_1}, p_{d_2}\}}$. Let $F$ and $F_M$ be as defined in \eqref{Introducing theta in multiplier}. Then
\begin{align*}
    \| F_M(\mathcal{L}, T) f\|_{L^2(G) } &\leq C 2^{-M d_2(1/p-1/2)} \,\|F\|_{L^2} ^{1-\theta_p}\,  \|F\|_{2^{M},2 } ^{\theta_p} \, \|f\|_{L^p(G)},
\end{align*}
where $\theta_p \in [0,1]$  satisfies $1/p = (1- \theta_p) + \theta_p/\min\{p_{d_1}, p_{d_2} \}. $
\end{proposition}


Using (\ref{equation: cowling-sikora norm relation}), the following result follows immediately from Proposition \ref{Theorem: Truncated Restriction Estimate strong form}.
\begin{corollary}\label{Inequality: Truncated restriction equation weak form I}
    Suppose $1 \leq p \leq \min{\{p_{d_1}, p_{d_2}\}}$. Let  $F$ and $F_M$ be as in \eqref{Introducing theta in multiplier}. Then
\begin{align*}
    \| F_M(\mathcal{L}, T) f\|_{L^2(G) } \leq C 2^{-M d_2(1/p-1/2)} \|F\|_{2^{M},2 } \, \|f\|_{L^p(G)} .
\end{align*}
\end{corollary}

The following proposition, which we sometimes call weighted Plancherel estimate with respect to the first layer, plays an important role later in our proof.
\begin{proposition}\cite[Proposition 8.1]{Niedorf_Metivier_group_2023}
\label{Theorem : First layer weighted Plancherel}
Let $\mathcal{K}_{F_M(\mathcal{L}, T)}$ denote the convolution kernel of $F_M(\mathcal{L}, T)$. Then for any $\alpha \geq 0$, we have 
\begin{align*}
 \int_G \left| |x|^{\alpha} \mathcal{K}_{F_M(\mathcal{L}, T)} (x, u) \right|^2 \, d(x, u) &\leq C\ 2^{M(2\alpha -d_2)} \| F\|_{L^2} ^2.
\end{align*}
\end{proposition}

In the sequel, we also require the pointwise weighted kernel estimate.  Let us set $X=(X_1, \ldots, X_{d_1}, T_1, \ldots, T_{d_2})$.
\begin{proposition}
\label{Proposition: Weighted Plancherel with L infinity condition}
Let $\Gamma \in \mathbb{N}^{d_1 + d_2}$. Then for all bounded Borel functions $F : \mathbb{R} \to \mathbb{C}$ supported in $[0,R]$ and for all $\beta \geq 0$, whenever $s> \beta $ we have
\begin{align}
\label{Pointwise Kernel estimate with weight}
    |(1+R \|(x,u)\|)^{\beta} \mathcal{K}_{F(\sqrt{\mathcal{L}})}(x,u) | &\leq C\, R^Q\, \|\delta_R F\|_{L^{\infty}_{s}} \\
   \text{and} \label{Pointwise Kernel estimate with weight second} \quad |(1+R \|(x,u)\|)^{\beta} X^{\Gamma} \mathcal{K}_{F(\sqrt{\mathcal{L}})}(x,u) | &\leq C\, R^{|\Gamma|}\, R^Q\, \|\delta_R F\|_{L^{\infty}_{s}} .
\end{align}
Moreover, we also have
\begin{align}
\label{Pointwise Kernel estimate without weight}
    |\mathcal{K}_{F(\sqrt{\mathcal{L}})}(x,u) | &\leq C\, R^Q\, \|\delta_R F\|_{L^{\infty}} .
\end{align}
\end{proposition}
\begin{proof}
Note that enough to prove the estimate \eqref{Pointwise Kernel estimate with weight second}, as the estimate \eqref{Pointwise Kernel estimate with weight} follows from \eqref{Pointwise Kernel estimate with weight second} by just taking $\Gamma=0$. The main ingredients to proof this result is the following Gaussian bound for the heat kernel and its gradient estimate.
    
It is known that for all $t>0$, the heat kernel $\mathcal{K}_{\exp{(-t\mathcal{L})}}$ satisfies the following estimates (see \cite{Varopoulos_Analysis_Lie_Group_1988}, \cite{Alexopoulos_Spectral_multi_poly_growth_1994}).
\begin{align*}
    |X^{\Gamma} \mathcal{K}_{\exp{(-t\mathcal{L})}}(x,u)| & \leq C t^{-|\Gamma|/2} t^{-Q/2} \exp{\left\{-c \tfrac{\|(x,u)\|^2}{t} \right\}} \quad \text{for} \quad t>0 .
\end{align*}
Then we have
\begin{align*}
    \int_{G \setminus B(0,r)} \left|X^{\Gamma} \mathcal{K}_{\exp(-t\mathcal{L})} (x,u) \right|^2 \ d(x,u) &\leq C t^{-|\Gamma|} t^{-Q} e^{-c\frac{r^2}{t}} \int_{G \setminus B(0,r)} \exp{\left\{-c \tfrac{\|(x,u)\|^2}{t} \right\}} \ d(x,u) \\
    &\leq C t^{-|\Gamma|} t^{-Q/2} e^{-c\frac{r^2}{t}} ,
\end{align*}
where we have used the fact $\int_{G} \exp{\{-c \tfrac{\|(x,u)\|^2}{t}\}} \ d(x,u) \leq C\, t^{Q/2} $. In particular, we get
\begin{align}
\label{L2 norm of heat kernel}
    \int_{G} \left|X^{\Gamma} \mathcal{K}_{\exp(-t\mathcal{L})} (x,u) \right|^2 \ d(x,u) &\leq C t^{-|\Gamma|} t^{-Q/2} .
\end{align}
We will use some ideas from \cite[Lemma 4.2, 4.3]{Bui_Duong_Spectral_multiplier__Trieble_Lizorkin_2021}. Let us set $F_1(\lambda)=F(\sqrt{\lambda}) e^{\lambda/R^2}$ and $F_2(\lambda)= e^{-\lambda/R^2}$. Then $F(\sqrt{\lambda})= F_1(\lambda) F_2(\lambda)$ and therefore we can write
\begin{align}
\label{Kernel in terms of convolution of to kernel}
    \mathcal{K}_{F(\sqrt{\mathcal{L}})}(x,u) &= \mathcal{K}_{F_1(\mathcal{L})} * \mathcal{K}_{F_2(\mathcal{L})}(x,u) \\
    &\nonumber = \int_{G} \mathcal{K}_{F_1(\mathcal{L})}((z,s)^{-1}(x,u)) \mathcal{K}_{F_2(\mathcal{L})}(z,s) \ d(z,s).
\end{align}
So that from \eqref{L2 norm of heat kernel} and \eqref{Kernel in terms of convolution of to kernel}, we get
\begin{align}
\label{L2 estimate of kernel with gradient}
    \int_{G} |X^{\Gamma} \mathcal{K}_{F(\sqrt{\mathcal{L}})}(x,u)|^2 \ d(x,u) &\leq C \|F_1(\mathcal{L})\|_{L^2 \to L^2}^2 \|X^{\Gamma} \mathcal{K}_{e^{-R^{-2}\mathcal{L}}}\|_{L^2}^2 \\
    &\nonumber \leq C \|F_1\|_{L^{\infty}}^2\, R^{2|\Gamma|}\, R^{Q} \\
    &\nonumber \leq C R^{2|\Gamma|}\, R^{Q}\, \|F\|_{L^{\infty}}^2.
\end{align}
Therefore, using H\"older's inequality, \eqref{L2 norm of heat kernel} and \eqref{L2 estimate of kernel with gradient} from \eqref{Kernel in terms of convolution of to kernel} we can see
\begin{align}
\label{Estimate: kernel without weight}
    & |X^{\Gamma} \mathcal{K}_{F(\sqrt{\mathcal{L}})}(x,u)| \\
    &\nonumber \leq \left(\int_{G} |\mathcal{K}_{F_1(\mathcal{L})}((z,s)^{-1}(x,u))|^2 \ d(z,s) \right)^{1/2} \left(\int_G |X^{\Gamma} \mathcal{K}_{e^{-R^{-2}\mathcal{L}}}(z,s)|^2 \ d(z,s)\right)^{1/2} \\
    &\nonumber \leq C R^{Q/2}\, \|F\|_{L^{\infty}}\, R^{|\Gamma|}\, R^{Q/2} \\
    &\nonumber \leq C R^{|\Gamma|}\, R^{Q}\, \|\delta_R F\|_{L^{\infty}} .
\end{align}
Now, let us set $G(\lambda)= F(\sqrt{R^2 \lambda}) e^{\lambda} $. Then we have $F(\sqrt{\lambda})= G(R^{-2} \lambda) e^{-R^{-2}\lambda}$. Therefore, by Fourier inversion formula we can write
\begin{align*}
    F(\sqrt{\lambda}) &= \frac{1}{2 \pi} \int_{\mathbb{R}} \widehat{G}(\tau) \exp{(-(1-i \tau)R^{-2} \lambda)} \  d\tau ,
\end{align*}
From this we can easily see that
\begin{align}
\label{Expressing kernel in terms of Fourir trnasform}
    \mathcal{K}_{F(\sqrt{\mathcal{L}})}(x,u) &= \frac{1}{2\pi} \int_{\mathbb{R}} \widehat{G}(\tau) \mathcal{K}_{\exp{(-(1-i \tau)R^{-2} \mathcal{L}})}(x,u) \ d\tau .
\end{align}
Then similarly as in \cite[Theorem 7.3]{Ouhabaz_Analysis_heat_equation_domain_2005}, one can show that
\begin{align*}
    |X^{\Gamma} \mathcal{K}_{\exp{(-(1-i \tau)R^{-2} \mathcal{L}})}(x,u)| & \leq C R^{|\Gamma|}\, R^Q\, \exp{\left\{-c \tfrac{R^2 \|(x,u)\|^2}{(1+\tau^2)} \right\}} .
\end{align*}
From this for any $\beta \geq 0$ we get
\begin{align}
\label{Gradient estimate of heat kernel with weight}
    |X^{\Gamma} \mathcal{K}_{\exp{(-(1-i \tau)R^{-2} \mathcal{L}})}(x,u)| (1+R \|(x,u)\|)^{\beta} &\leq C R^{|\Gamma|}\, R^Q\, (1+|\tau|)^{\beta} .
\end{align}
Therefore, from \eqref{Expressing kernel in terms of Fourir trnasform}, \eqref{Gradient estimate of heat kernel with weight} and using Cauchy-Schwartz inequality for some $\varepsilon>0$ we have
\begin{align}
\label{Estimate: kernel with weight}
    & |X^{\Gamma} \mathcal{K}_{F(\sqrt{\mathcal{L}})}(x,u)| (1+R \|(x,u)\|)^{\beta} \\
    &\nonumber \leq C \int_{\mathbb{R}} |\widehat{G}(\tau)| |X^{\Gamma} \mathcal{K}_{\exp{(-(1-i \tau)R^{-2} \mathcal{L}})}(x,u)| (1+R \|(x,u)\|)^{\beta} \ d\tau \\
    &\nonumber \leq C R^{|\Gamma|}\, R^Q\, \left
    (\int_{\mathbb{R}} |\widehat{G}(\tau)|^2 (1+|\tau|^2)^{\beta+1/2+\varepsilon} \ d\tau \right)^{1/2} \left(\int_{\mathbb{R}} (1+|\tau|^2)^{-1/2-\varepsilon} \ d\tau \right)^{1/2} \\
    &\nonumber \leq C R^{|\Gamma|}\, R^Q\, \|G\|_{L^2_{\beta+1/2+\varepsilon}} \\
    &\nonumber \leq C R^{|\Gamma|}\, R^Q\, \|\delta_R F\|_{L^{\infty}_{\beta+1/2+\varepsilon}} .
\end{align}
Then interpolating (\ref{Estimate: kernel without weight}) and (\ref{Estimate: kernel with weight}) as in \cite[Lemma 4.3]{Duong_Ouhabaz_Sikora_Sharp_Multiplier_2002} for any $\beta \geq 0$, we get
\begin{align}
\label{Required estimate for intepolation}
    |X^{\Gamma} \mathcal{K}_{F(\sqrt{\mathcal{L}})}(x,u)| (1+R \|(x,u)\|)^{\beta} &\leq C R^{|\Gamma|}\, R^Q\, \|\delta_R F\|_{L^{\infty}_{\beta+\varepsilon}} .
\end{align}
Indeed, consider the linear operator $K_R : L^{\infty}([0, 1]) \to L^{\infty}(G)$ given by
\begin{align*}
    K_R(\mathfrak{F}) &= X^{\Gamma} \mathcal{K}_{\delta_{1/R} \mathfrak{F}(\sqrt{\mathcal{L}})} .
\end{align*}
Note that from \eqref{Estimate: kernel without weight} we have
\begin{align*}
    \|K_R(\mathfrak{F})\|_{L^{\infty}([0,1]) \to L^{\infty}(G)} &\leq C R^{|\Gamma|+Q} .
\end{align*}
Let us set
\begin{align*}
    L^{\infty}_{\beta, R} &:= L^{\infty}(G, \mu_{\beta, R}) \qquad \text{where} \qquad d\mu_{\beta, R}(x,u) = (1+R \|(x,u)\|)^{\beta} d(x,u) .
\end{align*}
Then from \eqref{Estimate: kernel with weight} we have
\begin{align*}
    \|K_R(\mathfrak{F})\|_{L^{\infty}_{\beta+1/2+\epsilon}([0,1]) \to L^{\infty}_{\beta, R}} &\leq C R^{|\Gamma|+Q} .
\end{align*}
Consequently, by interpolation \cite[proof of Lemma 1.2]{Mauceri_Meda_Multipliers_Stritified_group_1990} for every $\theta \in (0,1)$, there exists a positive constant $C$ such that for all $\beta>0$, $\varepsilon'>0$,
\begin{align*}
    \|X^{\Gamma} \mathcal{K}_{\delta_{1/R} \mathfrak{F}(\sqrt{\mathcal{L}})}\|_{L^{\infty}_{\beta \theta, R}} &\leq C R^{|\Gamma|+Q} \|\mathfrak{F}\|_{L^{\infty}_{\beta \theta+\theta/2+\varepsilon'}([0,1])} .
\end{align*}
Taking $\beta'=\beta \theta$ and choosing $\theta$ sufficiently small we obtain
\begin{align}
\label{Estimate after interpolation}
    \|X^{\Gamma} \mathcal{K}_{\delta_{1/R} \mathfrak{F}(\sqrt{\mathcal{L}})}\|_{L^{\infty}_{\beta', R}} &\leq C R^{|\Gamma|+Q} \|\mathfrak{F}\|_{L^{\infty}_{\beta' +\varepsilon''}([0,1])} ,
\end{align}
for all $\beta'>0$ and $\varepsilon''>0$.

Then the required estimate \eqref{Required estimate for intepolation} follows immediately from the above estimate \eqref{Estimate after interpolation}.

This completes the proof of the estimate \eqref{Pointwise Kernel estimate with weight second}. Now the proof of \eqref{Pointwise Kernel estimate without weight} follows from \eqref{Estimate: kernel without weight} by taking $\Gamma=0$.
\end{proof}

\section{Boundedness of Bochner-Riesz commutator} \label{Section: proof of boundedness of commutator}

In this section, we prove Theorem \ref{Theorem: Multiplier for Commutator}, Theorem \ref{Theorem: Bochner-Riesz commutator on Metivier} and Theorem \ref{Theorem: Bochner-riesz commuator on Heisenber-type group}. Note that Theorem \ref{Theorem: Bochner-Riesz commutator on Metivier} is an easy consequence of Theorem \ref{Theorem: Multiplier for Commutator} by setting $F(\lambda)= (1- \lambda^2)_{+} ^{\alpha}$, $t=1$ and using the fact that $F\in L^2_{\beta}$ if and only if $\alpha > \beta - 1/2$. Therefore let us start with the proof of Theorem \ref{Theorem: Multiplier for Commutator}.

\begin{proof}[Proof of Theorem~\ref{Theorem: Multiplier for Commutator}]
    
By duality and interpolation argument, in order to prove Theorem \ref{Theorem: Multiplier for Commutator}, it is enough to consider the case $q\in (p,2)$. Let $\delta_R$ be as in (\ref{Dilation on metivier}). Then we have $\delta_{t^{-1}} \circ \sqrt{\mathcal{L}} \circ \delta_t = \sqrt{\mathcal{L}} $, and consequently $\delta_{t^{-1}} \circ F(\sqrt{\mathcal{L}}) \circ \delta_t = F(t\sqrt{\mathcal{L}}) $ for any bounded Borel function $F : \mathbb{R} \to \mathbb{C}$. Therefore we obtain
\begin{align}
\label{Reducing to the case t=1 case}
    \|[b, F(t\sqrt{\mathcal{L}})]\|_{L^p(G) \to L^p(G)} &= \|[b_t, F(\sqrt{\mathcal{L}})]\|_{L^p(G) \to L^p(G)} ,
\end{align}
where $b_t = \delta_t b$ and $\|b_t\|_{BMO^{\varrho}(G)} = \|b\|_{BMO^{\varrho}(G)}$. In view of the above relation \eqref{Reducing to the case t=1 case} it is enough to consider $t=1$.

We first break $F$ in the Fourier transform side. To this end,  we choose an even bump function $\eta$ supported in $[-1, -1/4] \cup [1/4, 1]$ such that $\sum_{i\in \mathbb{Z}} \eta (2^{-i}\lambda) = 1$ for all $\lambda \neq  0$. Set $\eta_0 (\lambda) = 1 - \sum_{i\geq 1} \eta(2^{-i} \lambda)$ and $\eta_i (\lambda) = \eta(2^{-i} \lambda), \,i \in \mathbb{N}$. Using this let us define
\begin{align*}
    F^{(i)}(\lambda) = \frac{1}{2\pi} \int_{\mathbb{R}} \eta_i (\tau) \hat{F}(\tau) \cos \lambda \tau\, d\tau, \quad \mbox{for}\,\,i \geq 0.
\end{align*}
Then via functional calculus, we have  
\begin{align}
\label{equation: breaking of F into fourier transform side}
    F(\sqrt{\mathcal{L}})f &= \sum_{i \geq 0} F^{(i)} (\sqrt{\mathcal{L}})f .
\end{align}
Note that from \cite[Lemma 4.4]{Niedorf_Metivier_group_2023}, the sub-Laplacian $\mathcal{L}$ satisfies the finite speed propagation property, therefore by \cite[Lemma 2.1]{Chen_Tian_Ward_Commutator_Bochner_Riesz_Elliptic_2021} we have $\supp{K_{F^{(i)}(\sqrt{L})}} \subseteq \mathcal{D}_{2^i}$. So that from Lemma (\ref{Deduction: commutator to operator}), choosing $r=2^i$ we get,
\begin{align*}
   & \|[b, F^{(i)} (\sqrt{\mathcal{L}})]f\|_{L^q(G)}^q \\
   & \leq C \sum_k \Bigl(\| \chi_{4B_k}\, (b - b_{B_0})\, F^{(i)} (\sqrt{\mathcal{L}}) \chi_{\widetilde{B}_k} f \|_{L^q(G)}^q + \| \chi_{4 B_k}\, F^{(i)} (\sqrt{\mathcal{L}}) (\chi_{\widetilde{B}_k}\, (b - b_{B_0}) f )\|_{L^q(G)}^q  \Bigr).
\end{align*}
Thus, in view of (\ref{equation: breaking of F into fourier transform side}), to prove Theorem (\ref{Theorem: Multiplier for Commutator}), it is enough to show that for some $\delta>0$ and for all $i\geq 0$,
\begin{equation}
\label{equation: inequality for F^i}
  \|[b, F^{(i)}(\sqrt{\mathcal{L}})]f\|_{L^q(G)} \leq C 2^{-\delta i} \|b\|_{BMO^{\varrho}(G)} \|F\|_{L^2_{\beta}(\mathbb{R})} \|f\|_{L^q(G)}.
\end{equation}

Since $\mathcal{L}$ is left-invariant, in view of the previous two inequalities, we are reduced to showing that
\begin{align}\label{equation: red1}
 \| \chi_{4 B_0}\, (b - b_{B_0})\, F^{(i)} (\sqrt{\mathcal{L}}) \chi_{\widetilde{B}_0} f \|_{L^q(G)} &\leq C 2^{-\delta i} \|b\|_{BMO^{\varrho}(G)} \|F\|_{L^2_{\beta}(\mathbb{R})} \|\chi_{\widetilde{B}_0} f\|_{L^q(G)} , 
\end{align}
 and 
 \begin{align}\label{equation: red2}
  \| \chi_{4 B_0}\, F^{(i)} (\sqrt{\mathcal{L}}) (\chi_{\widetilde{B}_0}\, (b - b_{B_0}) f )\|_{L^q(G)} &\leq C 2^{-\delta i} \|b\|_{BMO^{\varrho}(G)} \|F\|_{L^2_{\beta}(\mathbb{R})} \|\chi_{\widetilde{B}_0} f\|_{L^q(G)} ,
\end{align}
for some $\delta>0$, where $\Tilde{B}_{0,k}$ is the translation of $\Tilde{B}_k$ via $(-x_k,-u_k)$ to the point $(0,0)$ and by abuse of notation we write $\Tilde{B}_{0,k}$ as $\Tilde{B}_{0}$ (see Lemma \ref{Deduction: commutator to operator}).

\begin{proof}[Proof of \eqref{equation: red1}]
We first choose a bump function $\psi \in C_c ^{\infty}(\mathbb{R})$  with support in $[1/16, 4]$ and $\psi = 1$ in $[1/8, 2]$. Then
\begin{align}\label{equation: decomposition of S_1}
    & \| \chi_{4 B_0}\, (b - b_{B_0})\, F^{(i)} (\sqrt{\mathcal{L}}) \chi_{\widetilde{B}_0} f \|_{L^q(G)}^q \\
    & \nonumber \leq C \|\chi_{4 B_0}\, (b - b_{B_0})\, (\psi F^{(i)})(\sqrt{\mathcal{L}}) \chi_{\widetilde{B}_0} f\|_{L^q(G)}^q \\
    & \hspace{5cm} + C \|\chi_{4 B_0}\, (b - b_{B_0})\, ((1-\psi) F^{(i)}) (\sqrt{\mathcal{L}}) \chi_{\widetilde{B}_0} f\|_{L^q(G)}^q \nonumber.
\end{align}
Let us first focus our attention to the term
 \begin{align*}
     \| \chi_{4 B_0} (b- b_{B_0}) (\psi F^{(i)})(\sqrt{\mathcal{L}}) (\chi_{\widetilde{B}_0}f)  \|_{L^q(G)}.
 \end{align*}
We decompose the multiplier $\psi F^{(i)}$ by a dyadic decomposition in the $|\mu|$ variable. From \eqref{Definition: Cutoff function theta} and (\ref{Introducing theta in multiplier}), for $M \in \mathbb{Z}$, we define $F^{(i)}_{M} : \mathbb{R} \times \mathbb{R} \rightarrow \mathbb{C}$ given by 
\begin{align}
\label{equation: cutoff in the T variable}
   F^{(i)} _{M}(\kappa, \tau) =( \psi F^{(i)})(\sqrt{\kappa})\, \Theta( 2^M \tau)\quad \mbox{for}\,\, \kappa \geq 0  
\end{align}
and $F^{(i)}_{M}(\kappa, \tau) = 0$  else. Then we can write
\begin{align*}
    ( \psi F^{(i)})(\sqrt{\mathcal{L}}) = \sum_{M \in \mathbb{Z} } F^{(i)} _{M} (\mathcal{L}, T) .
\end{align*}
Recall that the function $\mu \mapsto b^{\mu}$ is homogeneous of degree $1$, therefore $b^{\mu} = |\mu| b^{\Bar{\mu}}$, where $\Bar{\mu} = |\mu|^{-1} \mu$. Now as we have $\sqrt{\lambda_{\mathbf{k}}^{\mu}} \in \supp{\psi F^{(i)}}$ and $2^{M} |\mu| \in \supp{\Theta}$, therefore
\begin{align}
    \label{Inequality: Vanishing of multiplier for l small}
        1 \sim \lambda_{\mathbf{k}}^{\mu} \geq \sum_{n=1}^{N} r_n b_n^{\mu} = |\mu| \sum_{n=1}^{N} r_n b_n^{\Bar{\mu}} \sim 2^{-M} \sum_{n=1}^{N} r_n b_n^{\Bar{\mu}} .
\end{align}
Note that from \cite[equation 2.7, Lemma 5]{Martini_Muller_New_Class_Two_Step_Stratified_Groups_2014} we have
\begin{align*}
        \sum_{n=1}^{N} r_n b_n^{\mu} \sim \left( \sum_{n=1}^{N} 2 r_n (b_n^{\mu})^2 \right)^{1/2} = (tr(J_{\mu}^* J_{\mu}))^{1/2}.
    \end{align*}
As $(tr(J_{\mu}^* J_{\mu}))^{1/2}$ is the pullback of the Hilbert-Schmidt norm via the injective map $\mu \mapsto J_{\mu}$, therefore $\sum_{n=1}^{N} r_n b_n^{\mu}$ does not vanish for all $\mu \neq 0$. Also from Proposition \ref{Prop: structural properties} all maps $\mu \mapsto b_n^{\mu}$ are continuous on $\mathfrak{g}_2^{*}$. As $\Bar{\mu} \in \{ \mu \in \mathfrak{g}_2^{*} : |\mu|=1\}$, from (\ref{Inequality: Vanishing of multiplier for l small}) we get $2^{-M} \lesssim 1$. Therefore there exists $l_0 \in \mathbb{N}$ such that
\begin{align*}
  ( \psi F^{(i)})(\sqrt{\mathcal{L}}) = \sum_{M \geq -l_0 } F^{(i)} _{M} (\mathcal{L}, T) .
\end{align*}
With the aid of this decomposition, we write
\begin{align}
\label{First decomposition of multiplier in theta}
   \chi_{4 B_0}  (b - b_{B_0}) ( \psi F^{(i)})(\sqrt{\mathcal{L}}) (\chi_{\widetilde{B}_0}f) &= \chi_{4 B_0}  (b - b_{B_0}) \Bigl( \sum_{M = -l_0 } ^{i} +  \sum_{M = i +1 } ^{\infty}  \Bigr) F^{(i)} _{M} (\mathcal{L}, T) (\chi_{\widetilde{B}_0}f) .
\end{align}
We distinguish the following cases. The case $M\geq i+1$ is easier to handle. Applying H\"older's inequality, we get
\begin{align*}
    &\| \chi_{4 B_0}  (b - b_{B_0}) \sum_{M = i + 1} ^{\infty} F^{(i)} _{M} (\mathcal{L}, T) (\chi_{\widetilde{B}_0}f) \|_{L^q(G)} \\
    & \leq \sum_{M = i + 1}^{\infty} \|\chi_{4 B_0}\, (b - b_{B_0})\|_{L^{\frac{2q}{2-q}}(G)} \| F^{(i)} _{M} (\mathcal{L}, T) (\chi_{\widetilde{B}_0}f)\|_{L^2(G)} .
\end{align*}
Then by (\ref{Cherecterisation of BMO}) and (\ref{Difference of average in terms of BMO}), we get
\begin{align}
\label{Inequality: BMO norm calculation} 
   \|\chi_{4 B_0}\, (b - b_{B_0})\|_{L^{\frac{2q}{2-q}}(G)} &\leq C \,\|b\|_{BMO^{\varrho}(G)}\,|B_0|^{\frac{2-q}{2q}}  .
 \end{align}
Therefore, using the above estimate we can see
\begin{align}\label{equation: brake 1}
     &\| \chi_{4 B_0}  (b - b_{B_0}) \sum_{M = i + 1}^{\infty} F^{(i)}_{M} (\mathcal{L}, T) (\chi_{\widetilde{B}_0}f)\|_{L^q(G)} \\
     &\nonumber \leq C \sum_{M = i + 1} ^{\infty} \|b\|_{BMO^{\varrho}(G)} 2^{i Q(1/q- 1/2)}\, \| F^{(i)} _{M} (\mathcal{L}, T) (\chi_{\widetilde{B}_0}f)\|_{L^2(G)} .
\end{align}
Now, applying Corollary \ref{Inequality: Truncated restriction equation weak form I} and (\ref{equation: cowling-sikora norm relation}), for $\beta > 1/2$ we get
\begin{align}\label{equation: brake 2}
    & \| F^{(i)} _{M} (\mathcal{L}, T) (\chi_{\widetilde{B}_0}f)\|_{L^2(G)}\\
    &\nonumber \lesssim 2^{-Md_2(1/p - 1/2)}\, \|\psi F^{(i)} \|_{2^M, 2}\, \|\chi_{\widetilde{B}_0}f\|_{L^p(G)}\\
    &\nonumber \lesssim 2^{-Md_2(1/p - 1/2)}\, 2^{iQ(1/p- 1/q)} \Bigl( \| F^{(i)}\|_{L^2} + 2^{-M\beta}\| F^{(i)}\|_{L^2 _{\beta}} \Bigr)\|\chi_{\widetilde{B}_0}f\|_{L^q(G)} .
\end{align}
Substituting the estimate (\ref{equation: brake 2}) into (\ref{equation: brake 1}), leads us
\begin{align}
\label{equation: estimate of remainder term of S_{11}}
     & \| \chi_{4 B_0}  (b - b_{B_0}) \sum_{M = i + 1} ^{\infty} F^{(i)} _{M} (\mathcal{L}, T) (\chi_{\widetilde{B}_0}f)\|_{L^q(G)} \\
     &\nonumber \lesssim \|b\|_{BMO^{\varrho}(G)} 2^{id(1/p- 1/2)} \bigl(\| F^{(i)}\|_{L^2} + 2^{-i\beta}\| F^{(i)}\|_{L^2_{\beta}}  \bigr)\, \|\chi_{\widetilde{B}_0}f\|_{L^q(G)}\\
     &\nonumber \lesssim \|b\|_{BMO^{\varrho}(G)} 2^{i(d(1/p- 1/2)- \beta)}\, \| F\|_{L^2_{\beta}}\,\|\chi_{\widetilde{B}_0}f\|_{L^q(G)}\\
     &\nonumber \lesssim \|b\|_{BMO^{\varrho}(G)} 2^{-i \varepsilon}\, \| F\|_{L^2_{\beta}}\,\|\chi_{\widetilde{B}_0}f\|_{L^q(G)},
\end{align}
provided $ \beta > d(1/p - 1/2)$.

Note that for $1\leq p \leq p_{d_2}$, we have $d(1/p - 1/2) > 1/2$. Therefore, for $\beta > d(1/p - 1/2)$, we obtain
\begin{align}
\label{Inequality: Error part for truncation}
     \left\| \chi_{4 B_0}  (b - b_{B_0}) \sum_{M = i + 1} ^{\infty} F^{(i)} _{M} (\mathcal{L}, T) (\chi_{\widetilde{B}_0}f) \right\|_{L^q(G)} &\lesssim \|b\|_{BMO^{\varrho}(G)} 2^{-i \delta}\, \| F\|_{L^2_{\beta}}\,\|\chi_{\widetilde{B}_0}f\|_{L^q(G)}, 
\end{align}
for some $\delta >0$.

Now we deal with the case $-l_0 \leq M \leq i$. Using the fact \eqref{Decomposition of ball into Euclidean balls}, for each $M$, first we decompose $\supp{(\chi_{\widetilde{B}_0}f)} \subseteq B\bigr((0, 0), \tfrac{2^i}{10} \bigl)$ in the first layer as 
\begin{align}
\label{Equation: Decomposition of ball into smaller balls}
    \supp{(\chi_{\widetilde{B}_0}f)} &= \bigcup_{m= 1} ^{N_M} B_m ^{M}, 
\end{align}
where $B_m ^{M} \subseteq B^{|\cdot|} (x_m^{M}, C \tfrac{2^M}{10}) \times B^{|\cdot|}(0, C \tfrac{2^{2i}}{100})$ are disjoint subsets and $|x_m^{M}-x_{m'}^{M}|> 2^M/20$ for $m \neq m'$. Note that $N_M$ is bounded by some constant times $2^{(i-M)d_1}$. For $\gamma>0$, we also define
\begin{align*}
    \Tilde{B}_m^{M} :=  B^{|\cdot|}(x_m ^{M}, 2C\, \tfrac{2^M}{10}\, 2^{i\gamma}) \times B^{|\cdot|}(0, 4 C \tfrac{2^{2i}}{100}) .
\end{align*}
Then the number of overlapping balls $N_{\gamma}$ can be bounded by $2^{C \gamma i}$.

With the aid of this decomposition, we write
\begin{align}
\label{Decomposition of the operator in error and main part}
    & \chi_{4B_0} (b - b_{B_0}) \sum_{M= -l_0} ^i F^{(i)} _{M} (\mathcal{L}, T) \chi_{\widetilde{B}_0}f
    = \chi_{4B_0} (b - b_{B_0}) \sum_{M= -l_0} ^i \sum_{m= 1} ^{N_M} \chi_{\Tilde{B}_m ^{M}} F^{(i)}_{M} (\mathcal{L}, T) \chi_{B_m ^{M}}f \\
    &\nonumber \hspace{6cm} + \chi_{4B_0} (b - b_{B_0}) \sum_{M= -l_0} ^i \sum_{m= 1} ^{N_M} (1-\chi_{\Tilde{B}_m ^{M}}) F^{(i)}_{M} (\mathcal{L}, T) \chi_{B_m ^{M}}f .
\end{align}
For the second term in the right hand side, to estimate
\begin{align*}
  \left\|\chi_{4B_0} \sum_{M= -l_0} ^i \sum_{m= 1} ^{N_M} (1-\chi_{\Tilde{B}_m ^{M}}) F^{(i)}_{M} (\mathcal{L}, T) \chi_{B_m ^{M}}f \right\|_{L^p(G) \rightarrow L^2(G)},
\end{align*}
via Riesz-Thorin interpolation theorem, we interpolate between the $L^1(G) \to L^2(G)$ estimate and $L^2(G) \to L^2(G)$ estimate. Let us first start with $L^1(G) \to L^2(G)$ estimate.

Let $\mathcal{K}^{(i)} _{M}$ be the convolution kernel of $ F^{(i)} _{M} (\mathcal{L}, T)$. Then, by Minkowski's integral inequality,
\begin{align}\label{equation: pre L1-L1 estimate}
    & \left\|\chi_{4B_0} \sum_{M= -l_0} ^i \sum_{m= 1} ^{N_M} (1-\chi_{\Tilde{B}_m ^{M}}) F^{(i)}_{M} (\mathcal{L}, T) \chi_{B_m ^{M}}f \right\|_{ L^2(G)}\\
    & \nonumber\leq \sum_{M= -l_0} ^i \sum_{m= 1} ^{N_M} \int_{B_m ^{M}} |f(x', u')|\left(\int_{G} \chi_{4B_0}^2(1-\chi_{\Tilde{B}_m ^{M}})^2 |\mathcal{K}^{(i)} _{M} \bigl((x', u')^{-1}(x,u)\bigr)|^2\, d(x,u) \right)^{\frac{1}{2}}\, d(x',u').
\end{align}
Suppose $(x,u) \in \chi_{4B_0} (1-\chi_{\Tilde{B}_m ^{M}}) $ and $(x', u')\in B_m^{M}$, then we have $|x'-x| \gtrsim 2^M\, 2^{i\gamma}$. Now for $\Bar{N} \in \mathbb{N}$, by Proposition \ref{Theorem : First layer weighted Plancherel},  we have
\begin{align*}
  &\Bigl(\int_{G} \chi_{4B_0}^2(1-\chi_{\Tilde{B}_m ^{M}})^2 |\mathcal{K}^{(i)} _{M} \bigl((x', u')^{-1}(x,u)\bigr)|^2\, d(x,u) \Bigr)^{\frac{1}{2}} \\
  &\nonumber \leq C (2^M 2^{i\gamma})^{-\Bar{N}} \Bigl(\int_{G}\, \left| |x-x'|^{\Bar{N}} \mathcal{K}^{(i)} _{M} \bigl(x-x', u-u'-\tfrac{1}{2}[x',x] \bigr)\right|^2\, d(x,u) \Bigr)^{\frac{1}{2}}\, \\
  & \nonumber\leq C (2^M 2^{i\gamma})^{-\Bar{N}} \Bigl(\int_{G}\, \left| |x|^{\Bar{N}} \mathcal{K}^{(i)} _{M} (x, u)\right|^2\, d(x,u) \Bigr)^{\frac{1}{2}}\, \\
  & \nonumber\lesssim  (2^M 2^{i\gamma})^{-\Bar{N}} \, 2^{M(\Bar{N}-d_2/2)}\, \|\psi F^{(i)}\|_{L^2} .
\end{align*}
Using this estimate in (\ref{equation: pre L1-L1 estimate}), we get
\begin{align}\label{eqquation: L1-L1 estimate}
     \left\|\chi_{4B_0} \sum_{M= -l_0} ^i \sum_{m= 1} ^{N_M} (1-\chi_{\Tilde{B}_m ^{M}}) F^{(i)}_{M} (\mathcal{L}, T) \chi_{B_m ^{M}}f \right\|_{ L^2(G)} &\lesssim \| F^{(i)}\|_{L^2} 2^{-i\gamma \Bar{N}}\, 2^{i \varepsilon}\,\|\chi_{\widetilde{B}_0}f\|_{L^1(G)}.
\end{align}
On the other hand, for $L^2(G) \to L^2(G)$ estimate, by Plancherel theorem we have
\begin{align*}
    \left\|\chi_{4B_0} \sum_{M= -l_0} ^i \sum_{m= 1} ^{N_M} (1-\chi_{\Tilde{B}_m ^{M}}) F^{(i)}_{M} (\mathcal{L}, T) \chi_{B_m ^{M}}f \right\|_{ L^2(G)} & \leq \|\psi F^{(i)}\|_{L^{\infty}} \sum_{M= -l_0} ^i \sum_{m= 1} ^{N_M} \|\chi_{B_m ^{M}}f \|_{L^2(G)}.
\end{align*}
Note that for $s>1/2$, Sobolev embedding theorem gives $\| F^{(i)}\|_{L^{\infty}} \leq \| F^{(i)}\|_{L^2 _{s}} \sim 2^{i s} \| F^{(i)}\|_{L^2}$. Then using Cauchy-Schwartz inequality together with the upper bound of $N_M$, we get
\begin{align}\label{equation: L2-L2}
  & \left\|\chi_{4B_0} \sum_{M= -l_0} ^i \sum_{m= 1} ^{N_M} (1-\chi_{\Tilde{B}_m ^{M}}) F^{(i)}_{M} (\mathcal{L}, T) \chi_{B_m ^{M}}f \right\|_{ L^2(G)}  \\
  &\nonumber \leq 2^{i s} \| F^{(i)}\|_{L^2} 2^{i \varepsilon} 2^{i d_1/2}\, \|\chi_{\widetilde{B}_0}f\|_{L^2(G)} \lesssim 2^{i s'} \| F^{(i)}\|_{L^2}\, \|\chi_{\widetilde{B}_0}f\|_{L^2(G)},
\end{align}
for $s'> s + d_1/2$.

Therefore, interpolating between the estimates (\ref{eqquation: L1-L1 estimate}) and (\ref{equation: L2-L2}), we have for $1\leq p < 2$,
\begin{align}
\label{Interpolation for all points}
    & \left\|\chi_{4B_0} \sum_{M= -l_0} ^i \sum_{m= 1} ^{N_M} (1-\chi_{\Tilde{B}_m ^{M}}) F^{(i)}_{M} (\mathcal{L}, T) \chi_{B_m ^{M}}f \right\|_{ L^2(G)} \\
    &\nonumber \lesssim 2^{i \varepsilon} 2^{-2i\gamma \Bar{N}(1/p - 1/2)}\, 2^{2i s' (1- 1/p)}\, \|F^{(i)}\|_{L^2}\, \|\chi_{\widetilde{B}_0}f\|_{L^p(G)} .
\end{align}
This estimate together with H\"older's inequality and (\ref{Inequality: BMO norm calculation}) gives
\begin{align}\label{equation: estimate of negligable part of S_{11}}
    & \left\|\chi_{4B_0} (b - b_{B_0}) \sum_{M= -l_0} ^i \sum_{m= 1} ^{N_M} (1-\chi_{\Tilde{B}_m ^{M}}) F^{(i)}_{M} (\mathcal{L}, T) \chi_{B_m ^{M}}f \right\|_{ L^q(G)}\\
    & \lesssim 2^{i \varepsilon} 2^{iQ(1/q- 1/2)}\,\|b\|_{BMO^{\varrho}(G)}\left\|\chi_{4B_0} \sum_{M= -l_0} ^i \sum_{m= 1} ^{N_M} (1-\chi_{\Tilde{B}_m ^{M}}) F^{(i)}_{M} (\mathcal{L}, T) \chi_{B_m ^{M}}f \right\|_{ L^2(G)}\nonumber\\
    & \lesssim 2^{i \varepsilon} 2^{iQ(1/q- 1/2)}\, \|b\|_{BMO^{\varrho}(G)}\,2^{-2i\gamma \Bar{N}(1/p - 1/2)}\, 2^{2i s' (1- 1/p)}\, \|F\|_{L^2}\, \|\chi_{\widetilde{B}_0}f\|_{L^p(G)}\nonumber\\
    & \nonumber \lesssim 2^{i \varepsilon} 2^{iQ(1/q- 1/2)}\, \|b\|_{BMO^{\varrho}(G)}\, 2^{-2i\gamma \Bar{N}(1/p - 1/2)}\, 2^{2i s' (1- 1/p)} \|F\|_{L^2}\, 2^{iQ(1/p- 1/q)}\|\chi_{\widetilde{B}_0} f\|_{L^q(G)} \nonumber\\
    &\lesssim 2^{-\delta i}\, \|b\|_{BMO^{\varrho}(G)}\,\|F\|_{L^2}  \|\chi_{\widetilde{B}_0} f\|_{L^q(G)} \nonumber,
\end{align}
for some $\delta >0$ by choosing $\Bar{N}$ large enough.

It remains to estimate the first summand of (\ref{Decomposition of the operator in error and main part}).  Choose $s$ such that $q<s<2$. Then by \eqref{Inequality: BMO norm calculation} for $L^{\frac{sq}{s-q}}(G)$ and application of H\"older's inequality yields
\begin{align}\label{equation: break1}
   & \left\| \sum_{M= -l_0} ^i \sum_{m= 1} ^{N_M} \chi_{4B_0} (b - b_{B_0}) \chi_{\widetilde{B}_m ^{M}} F^{(i)}_{M} (\mathcal{L}, T) \chi_{B_m ^{M}}f \right\|_{L^q(G)} ^q \\
   &\nonumber \leq C (i + 1 + l_0)^{q-1} N_{\gamma}^{q-1} \sum_{M= -l_0} ^i \sum_{m= 1} ^{N_M} \left\| \chi_{4B_0} (b - b_{B_0}) \chi_{\widetilde{B}_m ^{M}} F^{(i)}_{M} (\mathcal{L}, T) \chi_{B_m ^{M}}f \right\|_{L^q(G)} ^q\\
   &\nonumber \leq C 2^{iQ(1/q- 1/s)q}\, \|b\|_{BMO^{\varrho}(G)}^q (i + 1 + l_0)^{q-1} 2^{C \gamma j}  \sum_{M = -l_0} ^{i} \sum_{m=1} ^{N_M}  \|\chi_{\widetilde{B}_m ^{M}} F^{(i)}_{M} (\mathcal{L}, T) \chi_{B_m ^{M}}f\|_{L^s(G)} ^q\\
   &\nonumber \leq C 2^{iQ(1/q- 1/s)q}\, \|b\|_{BMO^{\varrho}(G)}^q 2^{i \varepsilon} 2^{C \gamma i}  \sum_{M = -l_0} ^{i} \sum_{m=1} ^{N_M}  \bigl( 2^M \,2^{\gamma i} \bigr)^{d_1(1/s -1/2)q}\, 2^{i2d_2(1/s - 1/2)q} \\
   &\nonumber\hspace{9cm} \, \| F^{(i)}_{M} (\mathcal{L}, T) \chi_{B_m ^{M}}f\|_{L^2(G)} ^q.
\end{align}
Now, using Proposition (\ref{Theorem: Truncated Restriction Estimate strong form}) in (\ref{equation: break1}), for $1\leq p \leq p_{d_2}$ (as $\min\{p_{d_1}, p_{d_2}\} = p_{d_2}$ for $d_1>d_2$), by choosing $q$ and $s$ very close, that is $1/q-1/s<\varepsilon$, we get
\begin{align}\label{equation: break2 process 1}
   & \left\|  \sum_{M= -l_0} ^i \sum_{m= 1} ^{N_M} \chi_{4B_0} (b - b_{B_0}) \chi_{\widetilde{B}_m ^{M}} F^{(i)}_{M} (\mathcal{L}, T) \chi_{B_m ^{M}}f \right\|_{L^q(G)} ^q\\  
   &\nonumber\lesssim 2^{iQ(1/q- 1/s)q}\, \|b\|_{BMO^{\varrho}(G)}^q 2^{i \varepsilon} 2^{C \gamma i} \sum_{M = -l_0} ^{i} \sum_{m=1} ^{N_M}  \bigl( 2^M\,2^{\gamma i} \bigr)^{d_1(1/s -1/2)q}\,2^{i2d_2(1/s - 1/2)q}  \\
   &\nonumber \hspace{4cm} 2^{-Md_2(1/p- 1/2)q}\,  \|\psi F^{(i)}\|_{L^2} ^{q(1- \theta_p)}\, \|\psi F^{(i)}\|_{2^M, 2} ^{q \theta_p}\, \|\chi_{B_m ^{M}}f\|_{L^p(G)}^q \\
   &\nonumber\lesssim 2^{iQ(1/q- 1/s)q} \|b\|_{BMO^{\varrho}(G)}^q 2^{i \varepsilon} 2^{C \gamma i} \sum_{M = -l_0} ^{i} \sum_{m=1} ^{M_l}  \bigl( 2^M\,2^{\gamma i} \bigr)^{d_1(1/s -1/2)q}\,2^{i2d_2(1/s - 1/2)q} \\
   &\nonumber \hspace{1cm}  2^{-Md_2(1/p- 1/2)q}\,  \|\psi F^{(i)}\|_{L^2} ^{q(1- \theta_p)}\, \|\psi F^{(i)}\|_{2^M, 2} ^{q \theta_p}\, 2^{Md_1(1/p - 1/q)q}\, 2^{i2d_2(1/p - 1/q)q}\,\|\chi_{B_m ^{M}}f\|_{L^q(G)} ^q\\
   & \nonumber\lesssim  2^{i \varepsilon }\,  2^{i 2d_2(1/p - 1/2)q} \,\|b\|_{BMO^{\varrho}(G)}^q \, 2^{C \gamma i} \, 2^{i\gamma d_1(1/s -1/2)q} \sum_{M = -l_0} ^{i} \|\psi F^{(i)}\|_{L^2} ^{q(1- \theta_p)}\, \|\psi F^{(i)}\|_{2^M, 2} ^{q \theta_p}\, \\
   & \nonumber\hspace{5cm} 2^{M d_1(1/p - 1/2)q}\, 2^{ -M d_2(1/p -  1/2)q} \|\chi_{\widetilde{B}_0}f\|_{L^q(G)} ^q, 
\end{align}
where $\theta_p = p_{d_2} ' (1-1/p)$ and $p_{d_2} '$ denotes the conjugate exponent of $p_{d_2}$.

Set $1/r = 1/p -1/2$ and  $\nu =\frac{1}{3} (\beta -d/r)$. Since we have $\|F^{(i)}\|_{L^2} \sim 2^{-\beta i} \|F^{(i)}\|_{L^2_{\beta}}= 2^{-3i\nu} 2^{-i d/r} \|F^{(i)}\|_{L^2_{\beta}}$, therefore using (\ref{equation: cowling-sikora norm relation}), for $\Tilde{s} > 1/2$, we have
\begin{align}\label{equation: discrete norm estimate process 1}
    \|\psi F^{(i)}\|_{L^2} ^{(1- \theta_p)}\, \|\psi F^{(i)}\|_{2^M, 2} ^{\theta_p}\, &\lesssim \| F^{(i)}\|_{L^2} ^{(1- \theta_p)} 2^{(i-M) \Tilde{s} \theta_p} \|F^{(i)}\|_{L^2} ^{\theta_p} \\
    &\nonumber \lesssim 2^{(i-M)\Tilde{s} \theta_p}\, 2^{-3i\nu} 2^{-i d/r} \|F\|_{L^2_{\beta}}.
\end{align}
Using this observation, and since $\varepsilon$ and $\gamma$ can be chosen arbitrary small, from (~\ref{equation: break2 process 1}) we get
\begin{align}\label{equation: main term process 1}
   & \left\|  \sum_{M= -l_0} ^i \sum_{m= 1} ^{N_M} \chi_{4B_0} (b - b_{B_0}) \chi_{\widetilde{B}_m ^{M}} F^{(i)}_{M} (\mathcal{L}, T) \chi_{B_m ^{M}}f \right\|_{L^q(G)} ^q\\  
   & \nonumber\lesssim 2^{i \varepsilon} 2^{i 2d_2q/r} \,\|b\|_{BMO^{\varrho}(G)}^q \, \sum_{M = -l_0} ^{i} 2^{(i-M)\Tilde{s}q \theta_p}\, 2^{-3i q\nu} 2^{-iq d/r} \|F\|_{L^2_{\beta}} ^q 2^{M d_1q/r}\, 2^{ -M d_2q/r} \|\chi_{\widetilde{B}_0}f\|_{L^q(G)} ^q \\
   & \nonumber \lesssim \, 2^{i \varepsilon} 2^{i 2d_2q/r} \|b\|_{BMO^{\varrho}(G)}^q 2^{-2iq \nu} \|F\|_{L^2_{\beta}} ^q\sum_{M = -l_0} ^{i} \bigl[2^{(M-i)(d_1 - d_2)}\, 2^{-(M-i)\Tilde{s} \theta_p r}\, 2^{(M-i)\nu r} \bigr]^{q/r} \\
   & \hspace{6cm} 2^{-M \nu q} 2^{i(d_1 - d_2)q/r}\, 2^{-idq/r}\, \|\chi_{\widetilde{B}_0}f\|_{L^q(G)} ^q \nonumber\\
   & \lesssim  2^{-2iq \nu} 2^{i \varepsilon}\, \|b\|_{BMO^{\varrho}(G)} ^q \|F\|_{L^2_{\beta}} ^q \sum_{M = -l_0} ^{i} \bigl[2^{(M-i)(d_1 - d_2 - \Tilde{s} \theta_p r + \nu r)} \bigr]^{q/r} \, \|\chi_{\widetilde{B}_0}f\|_{L^q(G)} ^q .\nonumber
\end{align}
Set $M-i=-\ell$. Choosing $\Tilde{s}$ sufficiently close to $1/2$, we see that the series 
\begin{align*}
    &  \sum_{\ell=0} ^{\infty} 2^{-\ell(d_1 - d_2 - \Tilde{s} \theta_p r + \nu r)q/r} 
\end{align*}
is convergent provided $d_1 - d_2 - \frac{r \theta_p}{2}   + \nu r >0$. Then for $1\leq p \leq 2$, the condition $d_1 - d_2 - \frac{r \theta_p}{2}   + \nu r >0$ is equivalent to 
\begin{align*}
   1\leq p \leq \frac{p_{d_2} ' + 2(d_1- d_2)}{p_{d_2} '+ d_1 - d_2}:=P_{d_1, d_2}.
\end{align*}
Now in order to find the required range of $p$ we make use of the Proposition \ref{prop: Radon-Hurwitz number} depending on the values of $(d_1, d_2)$. Therefore we make following two cases. 
\subsection*{Case (a) \texorpdfstring{$(d_1, d_2) \notin \{(4,3), (8,6), (8, 7)\}$}{}:}
By Proposition \ref{prop: Radon-Hurwitz number},  we have that $d_1 > 3d_2/2$. Which in turn implies that $ p_{d_2} \leq P_{d_1, d_2}$.
Therefore, choosing $\varepsilon>0$ very small we obtain that 
\begin{align}\label{equation: principal term}
 \left\|  \sum_{M= -l_0} ^i \sum_{m= 1} ^{M_l} \chi_{4B_0} (b - b_{B_0}) \chi_{\widetilde{B}_m ^{M}} F^{(i)}_{M} (\mathcal{L}, T) \chi_{B_m ^{M}}f \right\|_{L^q(G)} ^q   &\lesssim C 2^{-iq \nu}  \|F\|_{L^2_{\beta}}^q \|b\|_{BMO^{\varrho}(G)} ^q  \, \|\chi_{\widetilde{B}_0}f\|_{L^q(G)} ^q . 
\end{align}

Hence, aggregating the estimates obtained in (\ref{equation: estimate of remainder term of S_{11}}), (\ref{equation: estimate of negligable part of S_{11}}) and (\ref{equation: principal term}), we conclude that
\begin{align*}
     \left\|  \sum_{M= -l_0} ^i \sum_{m= 1} ^{N_M} \chi_{4B_0} (b - b_{B_0}) \chi_{\widetilde{B}_m ^{M}} F^{(i)}_{M} (\mathcal{L}, T) \chi_{B_m ^{M}}f \right\|_{L^q(G)} & \lesssim 2^{-i \delta} \|b\|_{BMO^{\varrho}(G)} \, \| F\|_{L^2_{\beta}}\,\|\chi_{\widetilde{B}_0}f\|_{L^q(G)} ,
\end{align*}
for some $\delta>0$.

\subsection*{Case (b) \texorpdfstring{$(d_1, d_2) \in \{(4,3), (8,6), (8, 7)\}$}{}:}
Note that $P_{4,3} =6/5$, $P_{8,6} = 17/12$, $P_{8,7} = 14/11$ and $p_{3} = 4/3$, $p_{6}= 14/9$, $p_{7}=8/5$. Therefore from (\ref{equation: main term process 1}) we conclude the following result.

For $(d_1, d_2) = (4,3)$, we have 
\begin{align}\label{equation: estimate for S11 at 4,3 process 1}
     \left\|  \sum_{M= -l_0} ^i \sum_{m= 1} ^{N_M} \chi_{4B_0} (b - b_{B_0}) \chi_{\widetilde{B}_m ^{M}} F^{(i)}_{M} (\mathcal{L}, T) \chi_{B_m ^{M}}f \right\|_{L^q(G)} & \lesssim 2^{-i \delta} \|b\|_{BMO^{\varrho}(G)} \, \| F\|_{L^2_{\beta}}\,\|\chi_{\widetilde{B}_0}f\|_{L^q(G)},
\end{align}
for some $\delta>0$, whenever $1\leq p \leq 6/5$.

For $(d_1, d_2) = (8,6)$, we have 
\begin{align*}
     \left\|  \sum_{M= -l_0} ^i \sum_{m= 1} ^{N_M} \chi_{4B_0} (b - b_{B_0}) \chi_{\widetilde{B}_m ^{M}} F^{(i)}_{M} (\mathcal{L}, T) \chi_{B_m ^{M}}f \right\|_{L^q(G)} & \lesssim 2^{-i \delta} \|b\|_{BMO^{\varrho}(G)} \, \| F\|_{L^2_{\beta}}\,\|\chi_{\widetilde{B}_0}f\|_{L^q(G)},
\end{align*}
for some $\delta>0$, whenever $1\leq p \leq 17/12$.

And for $(d_1, d_2) = (8,7)$, we have 
\begin{align*}
     \left\|  \sum_{M= -l_0} ^i \sum_{m= 1} ^{N_M} \chi_{4B_0} (b - b_{B_0}) \chi_{\widetilde{B}_m ^{M}} F^{(i)}_{M} (\mathcal{L}, T) \chi_{B_m ^{M}}f \right\|_{L^q(G)} & \lesssim 2^{-i \delta} \|b\|_{BMO^{\varrho}(G)} \, \| F\|_{L^2_{\beta}}\,\|\chi_{\widetilde{B}_0}f\|_{L^q(G)},
\end{align*}
for some $\delta>0$, whenever $1\leq p \leq 14/11$.

Now let us move our attention to the second summand of the right hand side of (\ref{equation: decomposition of S_1}). To this end, we first estimate the term 
\begin{align*}
 \|\chi_{4 B_0}\, (b - b_{B_0})\, ((1-\psi) F^{(i)}) (\sqrt{\mathcal{L}}) \chi_{\widetilde{B}_0} f\|_{L^q(G)}. 
\end{align*}
We further decompose the function $1 -\psi$ as follows. Since $1-\psi$ is supported outside the interval $(1/8,2)$, we can choose a function $\phi\in C_c ^{\infty} (2,8)$ such that
\begin{align*}
    1 - \psi(\xi) = \sum_{\iota \geq 0} \phi(2^{-\iota} \xi) + \sum_{\iota \leq -6}  \phi(2^{-\iota} \xi) =: \sum_{\iota \geq 0} \phi_{\iota} (\xi) + \sum_{\iota \leq -6} \phi_{\iota} (\xi), \quad \text{for all}\  \xi >0 .
\end{align*}
Therefore, by spectral theorem we can write
    \begin{equation}
        ((1 - \psi) F^{(i)})(\sqrt{\mathcal{L}}) =\sum_{\iota \geq 0} (\phi_{\iota}  F^{(i)}) (\sqrt{\mathcal{L}}) + \sum_{\iota \leq -6} (\phi_{\iota}  F^{(i)}) (\sqrt{\mathcal{L}})\nonumber.
\end{equation} 
Again, as in \eqref{equation: cutoff in the T variable}, define the function $F^{(i,\iota)}_{M,t}$ as
\begin{align*}
    F^{(i,\iota)}_{M} (\kappa, \tau) &= (\phi_{\iota} F^{(i)}) (\sqrt{\kappa}) \,\Theta( 2^M \tau)  \quad \text{for} \quad \kappa \geq 0,
\end{align*}
and $F^{(i,\iota)}_{M} (\kappa, \tau) = 0$ else. Then similar to \eqref{First decomposition of multiplier in theta} allows us to write
\begin{align}\label{equation: decomposition of S_{12}}
    & \chi_{4B_0}\, (b - b_{B_0})\, ((1-\psi) F^{(i)})(\sqrt{\mathcal{L}}) \chi_{\widetilde{B}_0} f \\
    &\nonumber=  \chi_{4 B_0}\, (b - b_{B_0})\, \Bigl(\sum_{\iota \geq 0} + \sum_{\iota \leq -6}\Bigr) \, \left( \sum_{M=-l_0 } ^i + \sum_{M=i+ 1 } ^{\infty}\right) F^{(i,\iota)}_{M} (\mathcal{L}, T) \chi_{\widetilde{B}_0} f \\
    &\nonumber=:\Bigl(\sum_{\iota \geq 0} + \sum_{\iota \leq -6}\Bigr) g_{j, \leq i} ^{\iota} + \Bigl(\sum_{\iota \geq 0} + \sum_{\iota \leq -6}\Bigr)g_{j, > i} ^{\iota} .
\end{align}
Notice that as $\supp{F} \subseteq [-1,1]$, $\supp{\phi} \subseteq [2,8] $ and $\check{\eta}$ is Schwartz class function, for any $\Bar{N}>0$ we have
\begin{align}
\label{Inequality: Calculation with cutoff phi}
 \|\phi_{\iota} F^{(i)}\|_{L^{\infty}} = 2^i \|\phi_{\iota} (F*\delta_{2^i} \check{\eta}) \|_{L^{\infty}} \leq C 2^{-\Bar{N}(i + \max{\{\iota,0}\})} \| F \|_{L^2} .
\end{align}
For the second term of (\ref{equation: decomposition of S_{12}}), applying Corollary \ref{Inequality: Truncated restriction equation weak form I} and using \eqref{discrete norm is dominated by sup norm} and \eqref{Inequality: Calculation with cutoff phi}, we get
\begin{align*}
    & \left\|\chi_{4 B_0} \Bigl(\sum_{\iota \geq 0} + \sum_{\iota \leq -6}\Bigr) \,  \sum_{M=i+ 1 } ^{\infty} F^{(i,\iota)}_{M} (\mathcal{L}, T) \chi_{\widetilde{B}_0} f  \right\|_{L^2} \\
    & \leq C \Bigl(\sum_{\iota \geq 0} + \sum_{\iota \leq -6}\Bigr) 2^{\iota Q(1/p - 1/2)}\, 2^{-i d_2 (1/p - 1/2)}\, \|(\phi_{\iota} F^{(i)})(2^{\iota+3} \cdot)\|_{2^M,2} \, \|\chi_{\widetilde{B}_0} f\|_{L^p} \\
    & \leq C \Bigl(\sum_{\iota \geq 0} + \sum_{\iota \leq -6}\Bigr) 2^{\iota Q(1/p - 1/2)}\, 2^{-i d_2 (1/p - 1/2)}\, \|\phi_{\iota} F^{(i)}\|_{L^{\infty}} \, \|\chi_{\widetilde{B}_0} f\|_{L^p}\\
    & \leq C \Bigl(\sum_{\iota \geq 0} + \sum_{\iota \leq -6}\Bigr) 2^{\iota Q(1/p - 1/2)}\, 2^{-i d_2 (1/p - 1/2)}\, 2^{-\Bar{N}(i + \max{\{\iota,0}\})} \| F \|_{2} \, \|\chi_{\widetilde{B}_0} f\|_{L^p} \\
    & \leq C 2^{-i d_2 (1/p - 1/2)}\, 2^{-\Bar{N} i}\|F\|_{L^{2}} \, \|\chi_{\widetilde{B}_0} f\|_{L^p} .
\end{align*}
This estimate together with H\"older's inequality and \eqref{Inequality: BMO norm calculation}, we get 
\begin{align}\label{equation: estimate for remainder term of S_{12}}
   & \left\|\Bigl(\sum_{\iota \geq 0} + \sum_{\iota \leq -6}\Bigr)g_{j, > i} ^{\iota} \right\|_{L^q}\\
   & \nonumber \leq  \|\chi_{4B_0}\, (b - b_{B_0})\|_{L^{\frac{2q}{2-q}}}  \left\| \chi_{4 B_0}\, \, \Bigl(\sum_{\iota \geq 0} + \sum_{\iota \leq -6}\Bigr) \,  \sum_{M=i+ 1 } ^{\infty} F^{(i,\iota)}_{M} (\mathcal{L}, T) \chi_{\widetilde{B}_0} f  \right\|_{L^2} \\
   & \nonumber \leq C \|b\|_{BMO^{\varrho}(G)} 2^{iQ(1/q - 1/2)}\, 2^{-i d_2 (1/p - 1/2)}\, 2^{-\Bar{N} i} \|F\|_{L^2} \, \|\chi_{\widetilde{B}_0} f\|_{L^p} \\
   & \nonumber \leq C \|b\|_{BMO^{\varrho}(G)} 2^{iQ(1/q - 1/2)}\, 2^{-i d_2 (1/p - 1/2)}\, 2^{-\Bar{N} i} \|F\|_{L^2} 2^{iQ(1/p- 1/q)} \|\chi_{\widetilde{B}_0} f\|_{L^q}\\
   & \nonumber \leq C 2^{i d(1/p - 1/2)}\, 2^{-\Bar{N} i} \|b\|_{BMO^{\varrho}(G)}  \|F\|_{L^2} \, \|\chi_{\widetilde{B}_0} f\|_{L^q}\\
   & \nonumber \leq C 2^{-i\delta} \|b\|_{BMO^{\varrho}(G)} \|F\|_{L^2} \, \|\chi_{\widetilde{B}_0} f\|_{L^q},
\end{align}
for some $\delta>0$, by choosing $\Bar{N}$ large enough.

Now, for the first term $\Bigl(\sum_{\iota \geq 0} + \sum_{\iota \leq -6}\Bigr) g_{j, \leq i} ^{\iota}$ of (\ref{equation: decomposition of S_{12}}) we use the decomposition carried out for the case $l\leq i$ in \eqref{Equation: Decomposition of ball into smaller balls}. Therefore we can write
\begin{align}\label{equation: finite propagation speed part}
    & \chi_{4 B_0}\, (b - b_{B_0})\, \sum_{M= -l_0} ^i \Bigl( \sum_{\iota \geq 0} + \sum_{\iota \leq -6}\Bigr) F^{(i, \iota)}_{M} (\mathcal{L}, T) \chi_{\Tilde{B}_0}f \\
    &\nonumber= \chi_{4 B_0}\, (b - b_{B_0})\, \Bigl( \sum_{\iota \geq 0} + \sum_{\iota \leq -6}\Bigr) \sum_{M= -l_0} ^i \sum_{m= 1}^{N_M} \bigl( \chi_{\Tilde{B}_m ^{M}} + (1-\chi_{\Tilde{B}_m ^{M}}) \bigr)  F^{(i, \iota)}_{M} (\mathcal{L}, T)\chi_{B_m ^{M}}f .
\end{align}
Let $\mathcal{K}^{(i, \,\iota)}_{M}$ be the convolution kernel of $ F^{(i, \iota)}_{M} (\mathcal{L}, T)$. Then as estimated in \eqref{equation: pre L1-L1 estimate}, an application of Minkowski's inequality together with Proposition \ref{Theorem : First layer weighted Plancherel} yields that
\begin{align*}
 & \left\|\sum_{M= -l_0} ^i \sum_{m= 1} ^{N_M}   (1-\chi_{\Tilde{B}_m ^{M}}) \,  F^{(i, \iota)}_{M} (\mathcal{L}, T)\chi_{B_m ^{M}}f \right\|_{ L^2(G)}\\
 & \leq \sum_{M= -l_0} ^i \sum_{m= 1} ^{N_M} \int_{B_m ^{M}} |f(x', u')|\Bigl(\int_{G}\, |\mathcal{K}^{(i, \iota)} _{M} \bigl((x', u')^{-1}(x,u)\bigr)|^2\, d(x,u) \Bigr)^{\frac{1}{2}}\, d(x',u')\\
  & \leq C \sum_{M= -l_0} ^i \sum_{m= 1} ^{N_M}  \int_{B_m^{M}} |f(x', u')|\, (2^{\iota Q/2}\, 2^{-Md_2/2}\, \|(\phi_{\iota} F^{(i)})(2^{\iota +3}\cdot)\|_{L^2} \, d(x',u')\\
   & \leq C \sum_{M= -l_0} ^i \sum_{m= 1} ^{N_M} \| \phi_{\iota} F^{(i)} \|_{L^{\infty}} \,2^{\iota Q/2}\,2^{-Md_2/2}\, \int_{B_m ^{M}} |f(x', u')| \, d(x',u')\\
   & \leq C 2^{i \varepsilon} \| \phi_{\iota} F^{(i)} \|_{L^{\infty}} \, 2^{\iota Q/2}\,\|\chi_{\Tilde{B}_0}f\|_{L^1(G)}.
\end{align*}
Moreover, similar to \eqref{equation: L2-L2} using Plancherel estimates and upper bound of $N_M$, 
\begin{align*}
    \left\|\sum_{M= -l_0} ^i \sum_{m= 1} ^{N_M} (1-\chi_{\Tilde{B}_m ^{M}}) \,  F^{(i, \iota)} _{M} (\mathcal{L}, T)\chi_{B_m ^{M}}f \right\|_{ L^2(G)} & \leq C \| \phi_{\iota} F^{(i)}\|_{L^{\infty}} 2^{i\varepsilon}  2^{id_1/2}\, \|\chi_{\Tilde{B}_0}f\|_{L^2(G)}.
\end{align*}
Then, by Riesz-Thorin interpolation theorem, for $1\leq p \leq 2$ we have
\begin{align*}
    & \left\|\sum_{M= -l_0} ^i \sum_{m= 1} ^{N_M} (1-\chi_{\Tilde{B}_m ^{M}}) \,  F^{(i, \iota)} _{M} (\mathcal{L}, T)\chi_{B_m ^{M}}f \right\|_{ L^2(G)} \\
    & \leq C  2^{\iota Q(1/p -1/2)}\, 2^{i d_1(1- 1/p)}\, 2^{i \varepsilon} \| \phi_{\iota} F^{(i)}\|_{L^{\infty}}\, \|\chi_{\Tilde{B}_0}f\|_{L^p(G)}.
\end{align*}

Using this estimate together with \eqref{Inequality: BMO norm calculation}, H\"older's inequality, (\ref{Inequality: Calculation with cutoff phi}) and again H\"older's inequality, we get
\begin{align}\label{equation: estimate for neglegible part of S_{12}}
    & \left\|\chi_{4 B_0}\, (b - b_{B_0}) \Bigl( \sum_{\iota \geq 0} + \sum_{\iota \leq -6}\Bigr) \sum_{M= -l_0} ^i \sum_{m= 1} ^{N_M} (1-\chi_{\Tilde{B}_m ^{M}}) \,  F^{(i, \iota)} _{M} (\mathcal{L}, T)\chi_{B_m ^{M}}f \right\|_{ L^q(G)}\\
    & \leq C 2^{i Q(1/q- 1/2)}\,\|b\|_{BMO^{\varrho}(G)} \Bigl( \sum_{\iota \geq 0} + \sum_{\iota \leq -6}\Bigr) \left\|\sum_{M= -l_0}^i \sum_{m= 1}^{N_M}   (1-\chi_{\Tilde{B}_m ^{M}}) \,  F^{(i, \iota)}_{M} (\mathcal{L}, T)\chi_{B_m ^{M}}f \right\|_{ L^2(G)}\nonumber\\
    & \leq C 2^{iQ(1/q- 1/2)}\, \|b\|_{BMO^{\varrho}(G)}  \Bigl( \sum_{\iota \geq 0} + \sum_{\iota \leq -6}\Bigr) 2^{\iota Q(1/p -1/2)}\, 2^{i d_1(1- 1/p)}\, 2^{i \varepsilon} \nonumber \\
    &\hspace{6cm} 2^{-\Bar{N}(i + \max{\{\iota,0}\})} \| F \|_{L^2}\, 2^{iQ(1/p- 1/q)}\|\chi_{\Tilde{B}_0} f\|_{L^q(G)}\nonumber\\
    & \leq C 2^{iQ(1/p- 1/2)}\, \|b\|_{BMO^{\varrho}(G)}  \,2^{i d_1(1- 1/p)}\, 2^{-\Bar{N}i}  \| F \|_{L^2}\, \|\chi_{\Tilde{B}_0} f\|_{L^q(G)}\nonumber\\
    &\leq C 2^{-\delta i}\, \|b\|_{BMO^{\varrho}(G)}\, \|F\|_{L^2} \|\chi_{\Tilde{B}_0} f\|_{L^q(G)}\nonumber,
\end{align}
for some $\delta >0$, by choosing $\Bar{N}$ large enough.

We are left with the estimation of
\begin{align}\label{equation: main part of S_{12}}
    \chi_{4 B_0}\, (b - b_{B_0})\, \Bigl( \sum_{\iota \geq 0} + \sum_{\iota \leq -6}\Bigr) \sum_{M= -l_0} ^i \sum_{m= 1}^{N_M} \chi_{\Tilde{B}_m ^{M}}  F^{(i, \iota)}_{M} (\mathcal{L}, T)\chi_{B_m ^{M}}f .
\end{align}
Choose $s$ such that $q<s<2$. Then using H\"older's inequality and \eqref{Inequality: BMO norm calculation}, we see that
\begin{align*}
    &\nonumber \left\| \chi_{4B_0}  (b - b_{B_0}) \sum_{M = -l_0} ^{i} \sum_{m=1} ^{N_M} \Bigl( \sum_{\iota \geq 0} + \sum_{\iota \leq -6}\Bigr) \chi_{\Tilde{B}_m ^{M}}  F^{(i, \iota)}_{M} (\mathcal{L}, T)\chi_{B_m ^{M}}f \right\|_{L^q(G)} ^q\\
    &\nonumber\leq C 2^{iQ(1/q- 1/s)q}\, \|b\|_{BMO^{\varrho}(G)}^q (i+1-l_0)^{q-1} N_{\gamma}^{q-1} \\
    &\nonumber \hspace{5cm} \sum_{M = -l_0} ^{i} \sum_{m=1} ^{N_M} \left\|\Bigl( \sum_{\iota \geq 0} + \sum_{\iota \leq -6}\Bigr) \chi_{\Tilde{B}_m ^{M}}  F^{(i, \iota)}_{M} (\mathcal{L}, T)\chi_{B_m ^{M}}f \right\|_{L^s(G)} ^q .
\end{align*}
Now applying H\"older's inequality, Corollary \ref{Inequality: Truncated restriction equation weak form I}, the fact \eqref{discrete norm is dominated by sup norm} and \eqref{Inequality: Calculation with cutoff phi} we get
\begin{align*}
   & \left\|\Bigl( \sum_{\iota \geq 0} + \sum_{\iota \leq -6}\Bigr) \chi_{\Tilde{B}_m ^{M}}  F^{(i, \iota)}_{M} (\mathcal{L}, T)\chi_{B_m ^{M}}f \right\|_{L^s(G)} \\
   &\nonumber\leq C \Bigl( \sum_{\iota \geq 0} + \sum_{\iota \leq -6}\Bigr)  \bigl( 2^l 2^{i \gamma} \bigr)^{d_1(1/s -1/2)} 2^{i2d_2(1/s - 1/2)}\, \| F^{(i, \iota)} _{M} (\mathcal{L}, T) \chi_{B_m ^{M}}f\|_{L^2(G)} \\
   &\nonumber\leq C \Bigl( \sum_{\iota \geq 0} + \sum_{\iota \leq -6}\Bigr) \bigl( 2^l 2^{i\gamma} \bigr)^{d_1(1/s -1/2)} (2^i)^{2d_2(1/s - 1/2)} 2^{\iota Q(1/p - 1/2)}\, 2^{-Md_2(1/p- 1/2)}\,  \\
   & \nonumber \hspace{10cm} \|\phi_{\iota} F^{(i)}\|_{2^M, 2} \,\|\chi_{B_m ^{M}}f\|_{L^p(G)} \\
   &\nonumber\leq C \Bigl( \sum_{\iota \geq 0} + \sum_{\iota \leq -6}\Bigr) \bigl( 2^l 2^{i\gamma} \bigr)^{d_1(1/s -1/2)} 2^{i2d_2(1/s - 1/2)} 2^{\iota Q(1/p - 1/2)}\, 2^{-ld_2(1/p- 1/2)} \\
   &\nonumber \hspace{4cm} \, 2^{-\Bar{N}(i + \max{\{\iota,0}\})} \| F \|_{L^2} 2^{Md_1(1/p - 1/q)}\, 2^{i2d_2(1/p - 1/q)}\,\|\chi_{B_m ^{M}}f\|_{L^q(G)} \\
   & \nonumber \leq C 2^{i \varepsilon} 2^{i \gamma d_1(1/s-1/2)} 2^{Md_1(1/p- 1/2+1/s-1/q)} 2^{-Md_2(1/p- 1/2)} \, 2^{-\Bar{N}i} 2^{i2d_2(1/p - 1/2+1/s-1/q)} \\
   & \hspace{11cm} \| F \|_{L^2} \|\chi_{\widetilde{B}_0} f\|_{L^q(G)} .
\end{align*}
Let us set $l- i = -\ell$ and recall that we have $d_1>d_2$ on M\'etivier groups. Therefore, from the previous two estimates by choosing $\gamma> 0$ sufficiently small and $s$ arbitrary close to $q$,
\begin{align}\label{equation: estimate for main term of S_{12}}
   & \left\| \chi_{4B_0}  (b - b_{B_0}) \sum_{M = -l_0} ^{i} \sum_{m=1} ^{M_l} \Bigl( \sum_{\iota \geq 0} + \sum_{\iota \leq -6}\Bigr) \chi_{\Tilde{B}_m ^{M}}  F^{(i, \iota)}_{M} (\mathcal{L}, T)\chi_{B_m ^{M}}f \right\|_{L^q(G)} ^q\\
   & \nonumber \leq C \|b\|_{BMO^{\varrho}(G)}^q 2^{i \varepsilon}  \| F \|_{L^2} ^q \sum_{M= -l_0} ^{i} 2^{(M-i)(d_1-d_2)(1/p- 1/2)q} \, 2^{i(d_1-d_2)(1/p- 1/2)q} \\
   &\nonumber \hspace{8cm} 2^{-\Bar{N} qi} 2^{i2d_2(1/p - 1/2)q}\, \|\chi_{\Tilde{B}_0} f\|_{L^q(G)} ^q\\
   & \nonumber \leq C 2^{i d (1/p -1/2)q}   \|b\|_{BMO^{\varrho}(G)}^q (i + 1 + l_0)^{q-1}  \| F \|_{L^2} ^q \sum_{\ell=0} ^{\infty} 2^{-\ell (d_1-d_2)(1/p- 1/2)q}  \, 2^{-\Bar{N} qi} \, \|\chi_{\widetilde{B}_0} f\|_{L^q(G)} ^q\\
    &\leq C 2^{-\delta iq}\, \|b\|_{BMO^{\varrho}(G)} ^q\, \|F\|_{L^2} ^q \|\chi_{\widetilde{B}_0} f\|_{L^q(G)} ^q \nonumber,
\end{align}
for some $\delta>0$, by choosing $\Bar{N}$ sufficiently large.

Finally, aggregating the estimates obtained in (\ref{equation: estimate for remainder term of S_{12}}), (\ref{equation: estimate for neglegible part of S_{12}}) and (\ref{equation: estimate for main term of S_{12}}), we get
\begin{align*}
     \| \chi_{4 B_0}\, (b - b_{B_0})\,((1-\psi) F^{(i)} (\sqrt{\mathcal{L}})) \chi_{\widetilde{B}_0} f \|_{L^q(G)} \lesssim 2^{- \delta i} \|b\|_{BMO^{\varrho}(G)} \, \|F\|_{L^2} \,\|\chi_{\widetilde{B}_0}f\|_{L^q(G)},
\end{align*}
This completes the proof of (\ref{equation: red1}).
\end{proof}
Using similar idea as of proof of \eqref{equation: red1} with obvious modification, one can also prove (\ref{equation: red2}) and therefore we omit the details.

This completes the proof of Theorem \ref{Theorem: Multiplier for Commutator} except for the point $(d_1, d_2)=(4,3)$. Note that at $(4,3)$, from \eqref{equation: estimate for S11 at 4,3 process 1} we only get the range $1\leq p \leq 6/5$. Therefore in order to get the full range of $p$ at $(4,3)$, that is for $1\leq p \leq 4/3$, one needs further analysis at that particular point, which we are going to explain now.

\subsection*{Improvement at \texorpdfstring{$(d_1, d_2)=(4,3)$}{}}

Here we will use ideas from \cite[Section 8]{Niedorf_Metivier_group_2023}. As the proof at this point is nearly same as the Theorem \eqref{Theorem: Multiplier for Commutator} for other points, so we will be brief. Following the argument of Theorem \eqref{Theorem: Multiplier for Commutator}, one can see that restriction on the range of $p$ comes from the estimate \eqref{equation: main term process 1}, which is deduced from the estimate \eqref{equation: break1}. Therefore from \eqref{equation: break1} we have
\begin{align}
\label{For exception point from Lq to Ls}
    & \left\| \sum_{M= -l_0} ^i \sum_{m= 1} ^{N_M} \chi_{4B_0} (b - b_{B_0}) \chi_{\widetilde{B}_m ^{M}} F^{(i)}_{M} (\mathcal{L}, T) \chi_{B_m ^{M}}f \right\|_{L^q(G)} ^q \\
    &\nonumber \leq C 2^{iQ(1/q- 1/s)q}\, \|b\|_{BMO^{\varrho}(G)}^q 2^{i \varepsilon} 2^{C \gamma j}  \sum_{M = -l_0} ^{i} \sum_{m=1} ^{N_M}  \|\chi_{\widetilde{B}_m ^{M}} F^{(i)}_{M} (\mathcal{L}, T) \chi_{B_m ^{M}}f\|_{L^s(G)}^q .
\end{align}
Now similarly as in the \cite[Page 32]{Niedorf_Metivier_group_2023}, first using the left invariance of $F^{(i)}_{M} (\mathcal{L}, T)$ we get $\|\chi_{\widetilde{B}_m ^{M}} F^{(i)}_{M} (\mathcal{L}, T) \chi_{B_m ^{M}}f\|_{L^s(G)} = \|\chi_{\widetilde{A}^{M}} F^{(i)}_{M} (\mathcal{L}, T) \chi_{A^{M}}f \|_{L^s(G)}$ and then decomposing $\chi_{A^{M}}$ in the second layer, the above quantity inside the $L^s$ norm we can write
\begin{align}
\label{Further decomposition into second layer}
    & \chi_{\widetilde{A}^{M}} F^{(i)}_{M} (\mathcal{L}, T) \chi_{A^{M}}f \\
    &\nonumber = \sum_{k=1}^{K_M} \chi_{\widetilde{A}_k^{M}} \chi_{\widetilde{A}^{M}} F^{(i)}_{M,t} (\mathcal{L}, T) \chi_{A_k^{M}}f + \sum_{k=1}^{K_M} (1-\chi_{\widetilde{A}_k^{M}}) \chi_{\widetilde{A}^{M}} F^{(i)}_{M} (\mathcal{L}, T) \chi_{A_k^{M}}f ,
\end{align}
where
\begin{align*}
    A^{M} \subseteq B^{|\cdot|} (0, C \tfrac{2^M}{10}) \times B^{|\cdot|}(0, C \tfrac{2^{2i}}{100}), \quad \quad \widetilde{A}^{M} :=  B^{|\cdot|}(0, 2C\, \tfrac{2^M}{10}\, 2^{i\gamma}) \times B^{|\cdot|}(0, 4 C \tfrac{2^{2i}}{100} 2^{i \gamma}) 
\end{align*}
and
\begin{align*}
    \widetilde{A}_k^{M} := \{(x,u) \in \widetilde{A}^{M} : \langle u-u_k^M , v \rangle < 4 \Bar{C} \tfrac{2^M 2^{i}}{100} 2^{i \gamma} \}
\end{align*}
with $v \in \mathfrak{g}_2$ and $|v|=1$ as in the \cite[Proposition 8.2]{Niedorf_Metivier_group_2023}.

Then from \eqref{For exception point from Lq to Ls} we have
\begin{align}
\label{for exceptional case Lq to Ls decomposition}
    & \left\| \sum_{M= -l_0} ^i \sum_{m= 1} ^{N_M} \chi_{4B_0} (b - b_{B_0}) \chi_{\widetilde{B}_m ^{M}} F^{(i)}_{M} (\mathcal{L}, T) \chi_{B_m ^{M}}f \right\|_{L^q(G)} ^q \\
    &\nonumber \leq C 2^{iQ(1/q- 1/s)q}\, \|b\|_{BMO^{\varrho}(G)}^q 2^{i \varepsilon} 2^{C \gamma j}  \sum_{M = -l_0} ^{i} \sum_{m=1} ^{N_M}  \Biggl[\left\|\sum_{k=1}^{K_M} \chi_{\widetilde{A}_k^{M}} \chi_{\widetilde{A}^{M}} F^{(i)}_{M} (\mathcal{L}, T) \chi_{A_k^{M}}f \right\|_{L^s(G)}^q \\
    &\nonumber \hspace{5cm} + \left\|\sum_{k=1}^{K_M} (1-\chi_{\widetilde{A}_k^{M}}) \chi_{\widetilde{A}^{M}} F^{(i)}_{M} (\mathcal{L}, T) \chi_{A_k^{M}}f \right\|_{L^s(G)}^q \Biggr] .
\end{align}
Regarding the second term inside the summation in \eqref{for exceptional case Lq to Ls decomposition}, using H\"older's inequality,
\begin{align*}
    & \left\|\sum_{k=1}^{K_M} (1-\chi_{\widetilde{A}_k^{M}}) \chi_{\widetilde{A}^{M}} F^{(i)}_{M} (\mathcal{L}, T) \chi_{A_k^{M}}f \right\|_{L^s(G)} \\
    &\leq \sum_{k=1}^{K_M} \bigl( 2^M \,2^{\gamma i} \bigr)^{d_1(1/s -1/2)}\, 2^{i2d_2(1/s - 1/2)} \|(1-\chi_{\widetilde{A}_k^{M}}) \chi_{\widetilde{A}^{M}} F^{(i)}_{M} (\mathcal{L}, T) \chi_{A_k^{M}}f \|_{L^2(G)} .
\end{align*}
To estimate this we need the following lemma.
\begin{lemma}\cite[Proposition 8.2]{Niedorf_Metivier_group_2023}
\label{second layer weighted Plancherel for exceptional}
Let $G$ be a M\'etivier group such that $d_1=4$ and $d_2=3$. $\mathcal{K}_{F_{M}(\mathcal{L}, T)}$ denote the convolution kernel of $F_{M}(\mathcal{L}, T)$. Then there exists a direction $v \in \mathfrak{g}_2$ with $|v|=1$ such that for any $\alpha \geq 0$ we have 
\begin{align*}
    \int_G \left| \langle u, v \rangle^{\alpha} \mathcal{K}_{F_{M,R}(\mathcal{L}, T)} (x, u) \right|^2 \, d(x, u) &\leq C\ 2^{M(2\alpha -d_2)} \| F\|_{L^2_{\alpha}}^2.
\end{align*}
\end{lemma}
Then with the help of Lemma \ref{second layer weighted Plancherel for exceptional}, similarly as in \eqref{Interpolation for all points} (also see \cite[Lemma 8.3]{Niedorf_Metivier_group_2023}) and using the fact $A_k^{M} \subseteq \widetilde{A}^{M}$, for $s'>1/2$ we get
\begin{align*}
    & \|(1-\chi_{\widetilde{A}_k^{M}}) \chi_{\widetilde{A}^{M}} F^{(i)}_{M} (\mathcal{L}, T) \chi_{A_k^{M}}f \|_{L^2(G)} \\
    &\lesssim C 2^{-2i\gamma \Bar{N}(1/p - 1/2)}\, 2^{2i s' (1- 1/p)}\, \|F^{(i)}\|_{L^2}\, \|\chi_{\widetilde{A}^{M}}f\|_{L^p(G)} \\
    &\lesssim C 2^{-2i\gamma \Bar{N}(1/p - 1/2)}\, 2^{2i s' (1- 1/p)}\, \|F^{(i)}\|_{L^2}\, 2^{iQ(1/p- 1/q)}\|\chi_{\widetilde{B}_0} f\|_{L^q(G)} ,
\end{align*}
where in the second line we have also used the fact $\|F^{(i)}\|_{L^2_{\Bar{N}}} \sim 2^{i \Bar{N}} \|F^{(i)}\|_{L^2}$.

Now using upper bound of $K_M \lesssim 2^{i-M} \leq 2^i$ we have
\begin{align*}
    & \left\|\sum_{k=1}^{K_M} (1-\chi_{\widetilde{A}_k^{M}}) \chi_{\widetilde{A}^{M}} F^{(i)}_{M} (\mathcal{L}, T) \chi_{A_k^{M}}f \right\|_{L^s(G)} \\
    &\leq C 2^i \bigl( 2^M 2^{\gamma i} \bigr)^{d_1(1/s -1/2)} 2^{i2d_2(1/s - 1/2)} 2^{-2i\gamma \Bar{N}(1/p - 1/2)} 2^{2i s' (1- 1/p)} \|F^{(i)}\|_{L^2}\, 2^{iQ(1/p- 1/q)}\|\chi_{\widetilde{B}_0} f\|_{L^q(G)} .
\end{align*}
Therefore by choosing $\Bar{N}$ large, one term in \eqref{for exceptional case Lq to Ls decomposition} can be estimated as follows.
\begin{align}
\label{Exceptional case error part}
    & 2^{iQ(1/q- 1/s)q}\, \|b\|_{BMO^{\varrho}(G)}^q 2^{i \varepsilon} 2^{C \gamma j}  \sum_{M = -l_0} ^{i} \sum_{m=1} ^{N_M} \left\|\sum_{k=1}^{K_M} (1-\chi_{\widetilde{A}_k^{M}}) \chi_{\widetilde{A}^{M}} F^{(i)}_{M} (\mathcal{L}, T) \chi_{A_k^{M}}f \right\|_{L^s(G)}^q \\
    &\nonumber \leq C 2^{iQ(1/q- 1/s)q}\, \|b\|_{BMO^{\varrho}(G)}^q 2^{i \varepsilon} 2^{C \gamma j}  \sum_{M = -l_0} ^{i} \sum_{m=1} ^{N_M} 2^{i q} \bigl( 2^M \,2^{\gamma i} \bigr)^{d_1(1/s -1/2)q}\, 2^{i2d_2(1/s - 1/2)q} \\
    &\nonumber \hspace{2cm} 2^{-2i\gamma \Bar{N}(1/p - 1/2) q}\, 2^{2i s' (1- 1/p) q}\, \|F^{(i)}\|_{L^2}^q \, 2^{iQ(1/p- 1/q)q}\|\chi_{\widetilde{B}_0} f\|_{L^q(G)}^q \\
    &\nonumber \leq C 2^{iQ(1/q- 1/s)q}\, \|b\|_{BMO^{\varrho}(G)}^q 2^{i \varepsilon} 2^{C \gamma j} 2^{i(1+d_1)} \bigl( 2^i  \,2^{\gamma i} \bigr)^{d_1(1/s -1/2)q}\, 2^{2 i d_2(1/s - 1/2)q} \\
    &\nonumber \hspace{2cm} 2^{-2i\gamma \Bar{N}(1/p - 1/2)q}\, 2^{2i s' (1- 1/p)q}\, \|F\|_{L^2}^q\, 2^{i Q(1/p- 1/q)q}\|\chi_{\widetilde{B}_0} f\|_{L^q(G)}^q \\
    &\nonumber \leq C 2^{-i \delta}\, \|b\|_{BMO^{\varrho}(G)}^q \| F\|_{L^2}^q \, \|\chi_{\widetilde{B}_0}f\|_{L^q(G)}^q ,
\end{align}
for some $\delta>0$.

On the other hand using the bounded overlapping property of the sets $\widetilde{A}_k^{M}$ we have
\begin{align*}
    \left\|\sum_{k=1}^{K_M} \chi_{\widetilde{A}_k^{M}} \chi_{\widetilde{A}^{M}} F^{(i)}_{M} (\mathcal{L}, T) \chi_{A_k^{M}}f \right\|_{L^s(G)}^q &\leq C 2^{C \gamma i} \left(\sum_{k=1}^{K_M} \| \chi_{\widetilde{A}_k^{M}} F^{(i)}_{M} (\mathcal{L}, T) \chi_{A_k^{M}}f\|_{L^s(G)}^s \right)^{q/s} .
\end{align*}
From H\"older's inequality, Corollary \ref{Inequality: Truncated restriction equation weak form I} and using $q<s$ for any $\Tilde{s}>1/2$ we get
\begin{align*}
    & \sum_{k=1}^{K_M} \| \chi_{\widetilde{A}_k^{M}} F^{(i)}_{M} (\mathcal{L}, T) \chi_{A_k^{M}}f\|_{L^s(G)}^s \\
    &\leq C \sum_{k=1}^{K_M} \bigl( 2^M \,2^{\gamma i} \bigr)^{(d_1+1)(1/s -1/2)s}\, 2^{i(2d_2-1)(1/s - 1/2)s} \, \| F^{(i)}_{M} (\mathcal{L}, T) \chi_{A_k^{M}}f\|_{L^2(G)}^s \\
    &\leq C \bigl( 2^M \,2^{\gamma i} \bigr)^{(d_1+1)(1/s -1/2)s}\, 2^{i(2d_2-1)(1/s - 1/2)s}\, 2^{-Md_2(1/p- 1/2)s}\, \|\psi F^{(i)}\|_{2^M, 2}^s \\
    &\hspace{6cm} 2^{M(d_1+1)(1/p - 1/q)s}\, 2^{i(2d_2-1)(1/p - 1/q)s}\,\sum_{k=1}^{K_M} \|\chi_{A_k^{M}}f\|_{L^q(G)}^s \\
    &\leq C \bigl( 2^M \,2^{\gamma i} \bigr)^{(d_1+1)(1/s -1/2)s}\, 2^{i(2d_2-1)(1/s - 1/2)s}\, 2^{-Md_2(1/p- 1/2)s}\, \|\psi F^{(i)}\|_{2^M, 2}^s \\
    &\hspace{4cm} 2^{M(d_1+1)(1/p - 1/q)s}\, 2^{i(2d_2-1)(1/p - 1/q)s}\, \left(\sum_{k=1}^{K_M} \|\chi_{A_k^{M}}f\|_{L^q(G)}^q \right)^{s/q} \\
    &\leq C \bigl( 2^M \,2^{\gamma i} \bigr)^{(d_1+1)(1/s -1/2)s}\, 2^{i(2d_2-1)(1/s - 1/2)s}\, 2^{-Md_2(1/p- 1/2)s}\, \\
    &\hspace{3cm} 2^{(i-M)\Tilde{s} s} \| F^{(i)}\|_{L^2}^s 2^{M(d_1+1)(1/p - 1/q)s}\, 2^{i(2d_2-1)(1/p - 1/q)s}\, \|\chi_{A^{M}}f\|_{L^q(G)}^s .
\end{align*}
Therefore the remaining term in \eqref{for exceptional case Lq to Ls decomposition}(the one term which not estimated above) we estimate as follows.
\begin{align*}
    & (2^i t)^{Q(1/q- 1/s)q}\, \|b\|_{BMO^{\varrho}(G)}^q 2^{i \varepsilon} 2^{C \gamma j}  \sum_{M = -l_0} ^{i} \sum_{m=1} ^{N_M} \left\|\sum_{k=1}^{K_M} \chi_{\widetilde{A}_k^{M}} \chi_{\widetilde{A}^{M}} F^{(i)}_{M} (\mathcal{L}, T) \chi_{A_k^{M}}f \right\|_{L^s(G)}^q \\
    &\leq C 2^{iQ(1/q- 1/s)q}\, \|b\|_{BMO^{\varrho}(G)}^q 2^{i \varepsilon} 2^{C \gamma j}  \sum_{M = -l_0} ^{i} \sum_{m=1} ^{N_M} \bigl( 2^M \,2^{\gamma i} \bigr)^{(d_1+1)(1/s -1/2)q}\, 2^{i(2d_2-1)(1/s - 1/2)q}\, \\
    &\hspace{2cm} 2^{-Md_2(1/p- 1/2)q}\,2^{(i-M)\Tilde{s} q} \| F^{(i)}\|_{L^2}^q 2^{M(d_1+1)(1/p - 1/q)q}\, 2^{i(2d_2-1)(\frac{1}{p} - \frac{1}{q})q}\, \|\chi_{B_m ^{M}}f\|_{L^q(G)}^q \\
    &\leq C 2^{i \varepsilon}\, \|b\|_{BMO^{\varrho}(G)}^q 2^{i \varepsilon} 2^{C \gamma j}  \sum_{M = -l_0} ^{i} 2^{M(d_1+1)(1/s -1/2)q}\, 2^{i(2d_2-1)(1/s - 1/2)q}\, \\
    &\hspace{1cm} 2^{-Md_2(1/p- 1/2)q}\,2^{(i-M)\Tilde{s} q} 2^{-i \beta q} \| F^{(i)}\|_{L^2_{\beta}}^q 2^{M (d_1+1)(1/p - 1/q)q}\, 2^{i(2 d_2-1)(1/p - 1/q)q}\, \|\chi_{\widetilde{B}_0}f\|_{L^q(G)}^q \\
    &\leq C 2^{i \varepsilon}\,2^{C \gamma j}\, \|b\|_{BMO^{\varrho}(G)}^q  \sum_{M = -l_0}^{i} 2^{M(1/p - 1/q+1/s -1/2)q+M d_1(1/p - 1/q+1/s -1/2)q-Md_2(1/p- 1/2)q-M \Tilde{s} q}\, \\
    &\hspace{1cm} 2^{-i(1/p - 1/q+1/s - 1/2)q+ i\Tilde{s} q +2 i d_2(1/p - 1/q+1/s - 1/2)q}\, 2^{-i d(1/p-1/2) q} \| F^{(i)}\|_{L^2_{d(1/p-1/2)}}^q \, \|\chi_{\widetilde{B}_0}f\|_{L^q(G)}^q \\
    &\leq C 2^{i \varepsilon}\,2^{C \gamma j}\, \|b\|_{BMO^{\varrho}(G)}^q \| F\|_{L^2_{d(1/p-1/2)}}^q \, \|\chi_{\widetilde{B}_0}f\|_{L^q(G)}^q  \sum_{M = -l_0}^{i} 2^{q(M-i)\{1+ (d_1-d_2)- \Tilde{s} \frac{2p}{2-p}\}(1/p -1/2)} .
\end{align*}
Note that for $1\leq p \leq p_{d_2}$ we have
\begin{align*}
    \frac{1}{p}-\frac{1}{2} \geq \frac{d_2+3}{2(d_2+1)}-\frac{1}{2} =\frac{1}{d_2+1} ,
\end{align*}
therefore at $(d_1, d_2)=(4,3)$ we get
\begin{align*}
    1+ (d_1-d_2)- \Tilde{s} \frac{2p}{2-p} \geq 1+ (d_1-d_2)- \Tilde{s} (d_2+1) = 2-4\Tilde{s} .
\end{align*}
Then for $\Tilde{s}>1/2$, choosing $\Tilde{s}$ very close to $1/2$ one can see
\begin{align*}
    \sum_{M = -l_0}^{i} 2^{q(M-i)\{1+ (d_1-d_2)- \Tilde{s} \frac{2p}{2-p}\}(1/p -1/2)} < \infty .
\end{align*}
As $\beta>d(1/p-1/2)$, choosing $\gamma>0$ very small yields
\begin{align}
\label{Exceptional case main part}
    & (2^i t)^{Q(1/q- 1/s)q}\, \|b\|_{BMO^{\varrho}(G)}^q 2^{i \varepsilon} 2^{C \gamma j}  \sum_{M = -l_0} ^{i} \sum_{m=1} ^{N_M} \left\|\sum_{k=1}^{K_M} \chi_{\widetilde{A}_k^{M}} \chi_{\widetilde{A}^{M}} F^{(i)}_{M} (\mathcal{L}, T) \chi_{A_k^{M}}f \right\|_{L^s(G)}^q \\
    &\nonumber \leq C 2^{-i \delta}\, \|b\|_{BMO^{\varrho}(G)}^q \| F\|_{L^2_{\beta}}^q \, \|\chi_{\widetilde{B}_0}f\|_{L^q(G)}^q ,
\end{align}
for some $\delta>0$.

Then combining the estimates of \eqref{Exceptional case error part} and \eqref{Exceptional case main part} from \eqref{for exceptional case Lq to Ls decomposition} we get
\begin{align*}
    & \left\| \sum_{M= -l_0} ^i \sum_{m= 1} ^{N_M} \chi_{4B_0} (b - b_{B_0}) \chi_{\widetilde{B}_m ^{M}} F^{(i)}_{M} (\mathcal{L}, T) \chi_{B_m ^{M}}f \right\|_{L^q(G)} \\
    &\nonumber \leq C 2^{-i \delta}\, \|b\|_{BMO^{\varrho}(G)} \| F\|_{L^2_{\beta}} \, \|\chi_{\widetilde{B}_0}f\|_{L^q(G)} .
\end{align*}
This completes the proof of Theorem \ref{Theorem: Multiplier for Commutator} for the point $(d_1, d_2)=(4,3)$.
\end{proof}

\begin{proof}[Proof of Theorem \ref{Theorem: Bochner-riesz commuator on Heisenber-type group}]
As the class of Heisenberg group is smaller than the class of M\'etivier groups and $p_{d_1, d_2} = p_{d_2} = \frac{2(d_2 + 1)}{d_2 + 3}$ for $(d_1, d_2) \notin \{(8,6), (8,7) \}$, proof of Theorem \eqref{Theorem: Bochner-riesz commuator on Heisenber-type group} follows from the proof of Theorem \eqref{Theorem: Bochner-Riesz commutator on Metivier} except for the point $(d_1, d_2) = \{(8,6), (8,7) \}$. We will prove Theorem \eqref{Theorem: Bochner-riesz commuator on Heisenber-type group} by proving Theorem \eqref{Theorem: Multiplier for Commutator} for the case of Heisenberg type group with $1 \leq p \leq 2(d_2 + 1)/(d_2 + 3)$. If one look carefully the proof of Theorem \eqref{Theorem: Multiplier for Commutator}, one can see that the restriction for the range of $p$ comes from the estimates \eqref{equation: break1}- \eqref{equation: main term process 1}, which occurs because the weaker type restriction estimate available for M\'etivier groups. Therefore it is clear that the presence of discrete norm $\|F\|_{2^M,2 }$ in the right hand side of Theorem \ref{Theorem: Truncated Restriction Estimate strong form} is the main obstacle from getting desired ranges of $p$ for the case $(d_1, d_2) \in \{(8,6), (8,7)\}$. If we can replace the discrete norm $\|F\|_{2^M,2 }$  with $\|F\|_{L^2(G)}$, we can improve the ranges of $p$. Such estimates are known to be true for Heisenberg type groups. Indeed we have the following result.
\begin{proposition}\cite[Theorem 3.2]{Niedorf_Spectral_Multiplier_Heisenber_Group_2024}
\label{Prop: restriction estimate for Heisenberg type}
    Let $\mathbb{H}$ be Heisenberg type group. Suppose $1 \leq p \leq 2(d_2 + 1)/(d_2 + 3)$. Then for $F$ and $F_{M}$ as defined in Section \eqref{section: restriction and kernel estimate}, we have
\begin{align*}
    \| F_{M}(\mathcal{L}, T) f\|_{L^2(G) } &\leq C 2^{-M d_2(1/p-1/2)} \,\|F\|_{L^2}\, \|f\|_{L^p(G)} .
\end{align*}
\end{proposition}
For Heisenberg type groups the eigen values are of the form $[k]|\mu|=(2k+d_1/2)|\mu|$. Although the cutoff considered in \cite[Theorem 3.2]{Niedorf_Spectral_Multiplier_Heisenber_Group_2024} (they have considered cutoff along $[k]:=2k+d_1/2$ variable) is different from the cutoff here we considered (along $|\mu|$ variable), because the support of $F$ lies away from origin, these two cutoffs are essentially give the same result. Indeed instead of taking $\Theta(2^{-M}[k])$ if we take $\Theta(2^{M}|\mu|)$, then the same conclusion hold as in \cite[Theorem 3.2]{Niedorf_Spectral_Multiplier_Heisenber_Group_2024}.

Therefore having Proposition \eqref{Prop: restriction estimate for Heisenberg type} in hand and following the arguments of the proof of Theorem \ref{Theorem: Multiplier for Commutator}, in the estimate of \eqref{equation: break2 process 1} instead of Proposition \eqref{Theorem: Truncated Restriction Estimate strong form}, here using Proposition \eqref{Prop: restriction estimate for Heisenberg type}, and choosing $\gamma$ small enough for $1 \leq p \leq 2(d_2 + 1)/(d_2 + 3)$ we get
\begin{align*}
    & \left\|  \sum_{M= 0} ^i \sum_{m= 1} ^{N_M} \chi_{4B_0} (b - b_{B_0}) \chi_{\Tilde{B}_m ^{M}} F^{(i)}_{M} (\mathcal{L}, T) \chi_{B_m ^{M}}f \right\|_{L^q(G)} ^q \\  
    & \lesssim  2^{i \varepsilon }\,  2^{i 2d_2(1/p - 1/2)q} \,\|b\|_{BMO^{\varrho}(G)}^q \, \|\psi F^{(i)}\|_{L^2}^q \, \|\chi_{\widetilde{B}_0}f\|_{L^q(G)} ^q \sum_{M = 0} ^{i} 2^{M (d_1-d_2)(1/p - 1/2)q} \\
    & \lesssim  2^{i \varepsilon }\,  2^{i d(1/p - 1/2)q}\, 2^{-i\beta q} \,\|b\|_{BMO^{\varrho}(G)}^q \, \|F\|_{L^2_{\beta}}^q \, \|\chi_{\widetilde{B}_0}f\|_{L^q(G)} ^q \sum_{M = 0} ^{i} 2^{(M-i) (d_1-d_2)(1/p - 1/2)q} \\
    &\lesssim 2^{-\delta i} \,\|b\|_{BMO^{\varrho}(G)}^q \, \|F\|_{L^2_{\beta}}^q \, \|\chi_{\widetilde{B}_0}f\|_{L^q(G)}^q ,
\end{align*}
for some $\delta>0$, as we have $\beta> d(1/p - 1/2)$, where in the last inequality we have used the fact that on Heisenberg type groups we always have $d_1>d_2$.

Hence, one can conclude that Theorem \eqref{Theorem: Multiplier for Commutator} holds on Heisenberg type groups for all $d_1, d_2 \geq 1$ with $1 \leq p \leq 2(d_2 + 1)/(d_2 + 3)$. 
\end{proof}

\section{Compactness of Bochner-Riesz commutator} \label{section: proof of copactness}

In this section, we  give the prove Theorem (\ref{Theorem: Compactness of Bochner-Riesz commutator}). In order to prove Theorem (\ref{Theorem: Compactness of Bochner-Riesz commutator}), we need the following result which characterizes the relatively compact subsets of $L^p(G)$, which is also known as Kolmogorov-Riesz compactness theorem.
\begin{proposition}\cite[Lemma 4.3]{Chen_Duong_Li_Wu_Compactness_Riesz_Transform_Stratified_group_2019}
\label{Proposition: Kolmogorov-Riesz compactness theorem}
Let $1 < p < \infty$ and $(x_0,u_0) \in G$. The subset $\mathcal{F}$ in $L^p(G)$ is relatively compact if and only if the following conditions are satisfied:
\begin{enumerate}
       \item $\displaystyle{ \sup_{f \in \mathcal{F}} \|f\|_{L^p (G)} < \infty}$;
       \item $\displaystyle{\lim_{R \rightarrow \infty} \int_{G\setminus B((x_0,u_0), R)}} |f(x,u)|^p\, d(x,u) = 0$ uniformly in $f\in \mathcal{F}$;
       \item $\displaystyle{\lim_{r \rightarrow 0} \int_{G} |f(x,u) - f_{B((x,u), r)}|^p\, d(x,u) = 0   }$ uniformly in $f\in \mathcal{F}$.
\end{enumerate}
\end{proposition}

With the above proposition in hand we are ready to proof the Theorem (\ref{Theorem: Compactness of Bochner-Riesz commutator}).

\begin{proof}[Proof of Theorem (\ref{Theorem: Compactness of Bochner-Riesz commutator})]
Let us set $F(\eta) = (1-\eta^2)_{+}^{\alpha}$. Notice that from \eqref{equation: breaking of F into fourier transform side} and (\ref{equation: inequality for F^i}), it follows that $\sum_{i =0}^{N} [b, F^{(i)}(\sqrt{\mathcal{L}})]$ converges to $[b, F(\sqrt{\mathcal{L}})]$ in the $L^q(G)$ norm. Then using the fact that compact operators are closed under the norm limit, the proof is reduced to show that $[b, F^{(i)}(\sqrt{\mathcal{L}})]$ is compact for all $i\geq 0$. Moreover, since $C_c ^{\infty} (G)$ is dense in $CMO^{\varrho}(G)$, we may restrict ourselves to the case when $b \in  C_c ^{\infty} (G)$. To this end, let $\mathcal{F}$ be an arbitrary bounded set in $L^q(G)$. To prove that $[b, F^{(i)}(\sqrt{\mathcal{L}})]$ is a compact operator, we have to show that the set $\mathcal{E}= \{ [b, F^{(i)}(\sqrt{\mathcal{L}})]f: f \in \mathcal{F}\}$ is relatively compact for each $i \geq 0$, that is $\mathcal{E}$ satisfies all the conditions (1), (2) and (3) of Proposition (\ref{Proposition: Kolmogorov-Riesz compactness theorem}).

The condition (1) is follows readily from (\ref{equation: inequality for F^i}) and the fact that $F\in L^2_{\beta}$ if and only if $\alpha > \beta - 1/2$. Indeed
\begin{align*}
   \sup_{f \in \mathcal{F}} \|[b, F^{(i)}(\sqrt{\mathcal{L}})]f \|_{L^q (G)}
    & \leq C \, \sup_{f \in \mathcal{F}} \|b\|_{BMO^{\varrho}(G)} \|F\|_{L^2 _{\beta}(\mathbb{R})} \|f\|_{L^q (G)} < \infty.
\end{align*}

Next, we proceed to verify remaining other two conditions of Proposition (\ref{Proposition: Kolmogorov-Riesz compactness theorem}). To this end, 
we choose a function $\psi\in C_c ^{\infty} (-4,4)$ such that $\psi(\lambda) = 1$ on $(-2,2)$ and $\varphi \in C_c^{\infty}(2,8)$. Then for all $\lambda>0$ we can decompose as follows:
\begin{align}\label{equation: decomposition of F^i}
  F^{(i)}(\lambda) = (\psi F^{(i)})(\lambda) + \sum_{\iota \geq 0} (\varphi_{\iota} F^{(i)})(\lambda), 
\end{align}
where $\phi_{\iota}(\lambda) = \varphi(2^{-\iota }\lambda)$ for $\iota>0$.

Now let us verify the condition (2). First note that it is enough to verify the condition (2) for all $C_c^{\infty}(G)$ functions in $\mathcal{F}$. Let $\epsilon>0$. Then for $f \in \mathcal{F}$, there exists $g \in C_c^{\infty}(G)$ such that $\|f-g\|_{L^q(G)} < \epsilon$. So that from \eqref{equation: inequality for F^i} for any $R>0$, we have
\begin{align}
\label{Compact commutator density reduction}
     & \|\chi_{G\setminus B(0,\, R)} [b, F^{(i)}(\sqrt{\mathcal{L}})]f \|_{L^q(G)} \\
     &\nonumber \leq  \|\chi_{G\setminus B(0,\, R)} [b, F^{(i)}(\sqrt{\mathcal{L}})](f-g) \|_{L^q(G)} +  \|\chi_{G\setminus B(0,\, R)} [b, F^{(i)}(\sqrt{\mathcal{L}})]g \|_{L^q(G)} \\
     &\nonumber \leq \epsilon + \|\chi_{G\setminus B(0,\, R)} [b, F^{(i)}(\sqrt{\mathcal{L}})]g \|_{L^q(G)} .
\end{align}
Therefore to check the condition (2), enough to show as $R \to \infty$,
\begin{align*}
    \|\chi_{G\setminus B(0,\, R)} [b, F^{(i)}(\sqrt{\mathcal{L}})]g \|_{L^q(G)} < \epsilon \quad \text{for} \quad g \in C_c^{\infty}(G) .
\end{align*}
As $b \in C_c^{\infty}(G)$, there exists $j\in \mathbb{N}$ such that $\supp{b} \subseteq B(0, 2^j)$. Take $R =  2\cdot 2^{2j}$. Then using the expression \eqref{equation: decomposition of F^i}, we decompose as follows.
\begin{align*}
   &  \|\chi_{G\setminus B(0,\, 2 \cdot 2^{2j})} [b, F^{(i)}(\sqrt{\mathcal{L}})]g \|_{L^q(G)} \\
   &\leq \|\chi_{G\setminus B(0,\, 2 \cdot 2^{2j})} [b, (\psi F^{(i)})(\sqrt{\mathcal{L}})]g \|_{L^q(G)} + \sum_{\iota \geq 0} \|\chi_{G\setminus B(0,\, 2 \cdot 2^{2j})} [b, (\varphi_{\iota} F^{(i)})(\sqrt{\mathcal{L}})]g \|_{L^q(G)} .
 \end{align*}
First we estimate the second term in the right hand side. Using the support of $b$ and Young's inequality
 \begin{align*}
     & \|\chi_{G\setminus B(0, 2 \cdot 2^{2j})} [b, (\varphi_{\iota} F^{(i)})(\sqrt{\mathcal{L}})]g \|_{L^q (G)} \\
     & = \Biggl\{  \int_{G\setminus B(0, 2 \cdot 2^{2j})} \Bigg | \int_{B(0, 2^j)} \mathcal{K}_{(\varphi_{\iota} F^{(i)})(\sqrt{\mathcal{L}})} ((x,u)^{-1}(y,t))\,  b(y,t)\, g(y,t)\, d(y,t)   \Bigg|^q d(x,u) \Biggl\}^{1/q} \\
     &\leq \Biggl\{  \int_{G} \Bigg | \int_{G} \chi_{G \setminus B(0, 2^{2j})}((x,u)^{-1}(y,t)) \mathcal{K}_{(\varphi_{\iota} F^{(i)})(\sqrt{\mathcal{L}})} ((x,u)^{-1}(y,t)) \\ 
     & \hspace{7cm} \chi_{B(0,2^j)}(y,t)  b(y,t)\, g(y,t)\, d(y,t)   \Bigg|^q d(x,u) \Biggl\}^{1/q} \\
     &\leq C \|b\|_{L^{\infty}} \|g\|_{L^q(G)} \int_{G\setminus B(0, 2^{2j})} | \mathcal{K}_{ (\varphi_{\iota} F^{(i)})(\sqrt{\mathcal{L}})} (x,u)|\, d(x,u) .
 \end{align*}
Note using the definition $F^{(i)}(\lambda) = 2^i (F*\delta_{2^i} \check{\eta})(\lambda)$ as in \eqref{Inequality: Calculation with cutoff phi}, for any fixed $s>0$,
\begin{align}
\label{equation: decay of Fi with phi}
     \|\varphi_{\iota} F^{(i)}(2^{\iota} \cdot)\|_{L^{\infty}_{s}} \leq C 2^{-\Bar{N}(i + \iota)} \| F \|_{L^2} ,
\end{align}
for any $\Bar{N}>0$. 

Then using Proposition (\ref{Proposition: Weighted Plancherel with L infinity condition}),  (\ref{equation: decay of Fi with phi}) and Lemma (\ref{lemma: outside distance}) we have 
\begin{align*}
   & \int_{G\setminus B(0, 2^{2j})} | \mathcal{K}_{ (\varphi_{\iota} F^{(i)})(\sqrt{\mathcal{L}})} (x,u)|\, d(x,u) \\
   & \leq C \int_{\|(x, u)\| >  2^{2j}} \frac{(1+2^{\iota} \|(x,u)\|)^{s} |\mathcal{K}_{ (\varphi_{\iota} F^{(i)})(\sqrt{\mathcal{L}})} (x,u)| }{(1+ 2^{\iota} \|(x,u)\|)^{s} }\, d(x,u) \\
   &\leq C 2^{\iota Q}\, \|\varphi_{\iota} F^{(i)}(2^{\iota} \cdot)\|_{L^{\infty} _{s+ \varepsilon}}\, \int_{\|(x, u)\| >  2^{2j}} \frac{d(x,u)}{(1+ 2^{\iota} \|(x,u)\|)^{s}} \\
   & \leq C 2^{\iota Q} 2^{-\Bar{N}(i+\iota)} \, 2^{-\iota s} 2^{2j(-s + Q)} ,
\end{align*}
for all $\Bar{N}> 0$, $\varepsilon>0$ and $s>Q$.

Then choosing $s>Q$ and $\Bar{N}> Q$, as $j \rightarrow \infty$ we see that 
\begin{align}
\label{Inequality: Condition 3 for phi F}
    \sum_{\iota \geq 0} \|\chi_{G\setminus B(0, 2 \cdot 2^{2j})} [b, (\varphi_{\iota} F^{(i)})(\sqrt{\mathcal{L}})]g \|_{L^q (G)} &\leq C \|b\|_{L^{\infty}} \|g\|_{L^q(G)} 2^{-2j(s - Q)} \sum_{\iota \geq 0} 2^{-\iota (\Bar{N}-Q)}  < \epsilon .
\end{align}
A similar argument also shows that as $j \to \infty$,
\begin{align}
\label{Inequality: Condition 3 for psi F}
    \|\chi_{G\setminus B(0, 2 \cdot 2^{2j})} [b, (\psi F^{(i)})(\sqrt{\mathcal{L}})]g \|_{L^q(G)} & <\epsilon ,
\end{align}
where for each $i \geq 0$, one have to use the following fact, which is a consequence of Sobolev embedding theorem
\begin{align}
\label{equation: F^{i} with psi}
    \|\psi F^{(i)}\|_{L^{\infty}_{s}} \leq C 2^{i(s+1/2+\epsilon)} \| F \|_{L^2} \leq C_{i,s, \epsilon} .
\end{align}
As we are proving compactness for each $i \geq 0$, so our constants $C_{i,s,\epsilon}$ may depend on $i$.

Thus combining estimates (\ref{Inequality: Condition 3 for phi F}), (\ref{Inequality: Condition 3 for psi F}) and (\ref{Compact commutator density reduction}), as $j \to \infty$ we conclude that
\begin{align*}
    \|\chi_{G\setminus B(0, 2 \cdot 2^{2j})} [b, F^{(i)}(\sqrt{\mathcal{L}})]f \|_{L^q(G)} < \epsilon.
\end{align*}
Not it only remains to check the condition (3) for $\mathcal{E}$. Once again using the density of $C_c^{\infty}(G)$ functions in $L^q(G)$ enough to check condition (3) for $f \in C_c^{\infty}(G)$. First we write 
\begin{align*}
    & [b, F^{(i)}(\sqrt{\mathcal{L}})]f(x,u) - \bigl([b, F^{(i)}(\sqrt{\mathcal{L}})]f\bigr)_{B((x,u),r)}\\
    & = \frac{1}{|B((x,u), r)| }\int_{B((x,u), r)} \left\{ [b, F^{(i)}(\sqrt{\mathcal{L}})]f(x,u)- [b, F^{(i)}(\sqrt{\mathcal{L}})]f(y,t) \right\}\, d(y,t) \\
    & = \frac{1}{|B(0, r)| } \int_{B(0, r)} \left\{ [b, F^{(i)}(\sqrt{\mathcal{L}})]f(x,u)- [b, F^{(i)}(\sqrt{\mathcal{L}})]f((x,u)(y,t)) \right\}\, d(y,t) ,
\end{align*}
and then for $\varepsilon>0$, we divide it into four parts as follows.
\begin{align*}
  & [b, F^{(i)}(\sqrt{\mathcal{L}})]f(x,u)- [b, F^{(i)}(\sqrt{\mathcal{L}})]f((x,u)(y,t)) \\
  &= \int_G \mathcal{K}_{F^{(i)}(\sqrt{\mathcal{L}})} ((x',u')^{-1}(x,u)) (b(x,u)-b(x',u')) f(x',u') d(x',u') \\
  & \hspace{1cm} - \int_G \mathcal{K}_{F^{(i)}(\sqrt{\mathcal{L}})} ((x',u')^{-1}(x,u)(y,t)) (b((x,u)(y,t))-b(x',u')) f(x',u') d(x',u') \\
  &=: J_1 + J_2 + J_3 + J_4 .
\end{align*}
where
\begin{align*}
    J_1 & = \int_{\|((x',u')^{-1}(x,u))\| >  \varepsilon^{-1} \|(y,t)\|}
    \mathcal{K}_{F^{(i)}(\sqrt{\mathcal{L}})} ((x',u')^{-1}(x,u)) \\ 
    & \hspace{5cm} (b(x,u)- b((x,u)(y,t))) f(x',u')\, d(x',u') ,
\end{align*}
\begin{align*}
    J_2 & =\int_{\|((x',u')^{-1}(x,u))\| >  \varepsilon^{-1} \|(y,t)\|} ( b((x,u)(y,t))- b(x',u')) \\
    & \hspace{2cm} (\mathcal{K}_{F^{(i)}(\sqrt{\mathcal{L}})} ((x',u')^{-1}(x,u))-  \mathcal{K}_{F^{(i)}(\sqrt{\mathcal{L}})} ((x',u')^{-1}(x,u)(y,t))) f(x',u')\, d(x',u') ,
\end{align*}
\begin{align*}
    J_3 &= \int_{\|((x',u')^{-1}(x,u))\| \leq  \varepsilon^{-1} \|(y,t)\|}
    \mathcal{K}_{F^{(i)}(\sqrt{\mathcal{L}})} ((x',u')^{-1}(x,u))  (b(x,u)- b(x',u')) f(x',u')\, d(x',u') ,
\end{align*}
and
\begin{align*}
    J_4 &= - \int_{\|((x',u')^{-1}(x,u))\| \leq  \varepsilon^{-1} \|(y,t)\|}
    \mathcal{K}_{F^{(i)}(\sqrt{\mathcal{L}})} ((x',u')^{-1}(x,u)(y,t)) \\ 
    & \hspace{5cm} (b((x,u)(y,t))- b(x',u')) f(x',u')\, d(x',u') .
\end{align*}

We now estimate each term $J_1$, $J_2$, $J_3$ and $J_4$ separately. For the term $J_1$, using \eqref{equation: decomposition of F^i} we first write
\begin{align*}
    |J_1| &\leq C \left[\int_{\|((x',u')^{-1}(x,u))\| >  \varepsilon^{-1} \|(y,t)\|} |\mathcal{K}_{(\psi F^{(i)})(\sqrt{\mathcal{L}})} ((x',u')^{-1}(x,u))| |f(x',u')|\, d(x',u') \right. \\
    & \hspace{1cm} + \left. \sum_{\iota \geq 0} \int_{\|((x',u')^{-1}(x,u))\| >  \varepsilon^{-1} \|(y,t)\|} |\mathcal{K}_{(\varphi_{\iota} F^{(i)})(\sqrt{\mathcal{L}})} ((x',u')^{-1}(x,u))| |f(x',u')|\, d(x',u') \right] \\
    & \hspace{8cm} \times |b(x,u)- b((x,u)(y,t))| \\
    &=:J_{11} + J_{12} .
\end{align*}
Let $0 < \tau < 1$. Then using mean value theorem together with Proposition (\ref{Proposition: Weighted Plancherel with L infinity condition}) and (\ref{equation: decay of Fi with phi}), for any $\Bar{N}>0$ we see that
\begin{align*}\label{estimate for J12}
    & |J_{12}| \\
    &\nonumber \leq C \|(y,t)\| \sum_{\iota \geq 0} \int_{\|((x',u')^{-1}(x,u))\| >  \varepsilon^{-1} \|(y,t)\|} |\mathcal{K}_{(\varphi_{\iota} F^{(i)})(\sqrt{\mathcal{L}})} ((x',u')^{-1}(x,u))| |f(x',u')|\, d(x',u') \\
    &\leq C \|(y,t)\| \sum_{\iota \geq 0} \sum_{k=0}^{\infty} \int_{2^{k} \varepsilon^{-1} \|(y,t)\|< \|((x',u')^{-1}(x,u))\| \leq 2^{k+1} \varepsilon^{-1} \|(y,t)\|}\nonumber \\
    & \hspace{2cm} \frac{(1+2^{\iota} \|(x',u')^{-1}(x,u)\| )^{Q+\tau} |\mathcal{K}_{(\varphi_{\iota} F^{(i)})(\sqrt{\mathcal{L}})} ((x',u')^{-1}(x,u))|}{(1+2^{\iota}\|(x',u')^{-1}(x,u)\|)^{Q+\tau}} |f(x',u')|\, d(x',u') \nonumber \\
    &\leq C \|(y,t)\| \sum_{\iota \geq 0} \sum_{k=0} ^{\infty} \frac{2^{\iota Q} \|\phi_{\iota} F^{(i)} (2^{\iota} \cdot)\|_{L^{\infty}_{Q + \tau + \epsilon} } |B((x,u), 2^{k+1} \varepsilon^{-1} \|(y,t)\|)|  }  {\{2^{\iota} 2^k \varepsilon^{-1} \|(y,t)\|\}^{Q + \tau}} \times\nonumber \\
    & \hspace{3cm} \frac{1}{|B((x,u), 2^{k+1} \varepsilon^{-1} \|(y,t)\|)| }\, \int_{\varrho((x',u'),(x,u)) \leq 2^{k+ 1} \varepsilon^{-1} \|(y,t)\|} |f(x',u')|\, d(x',u') \nonumber\\
    & \leq C \|(y,t)\| \sum_{\iota \geq 0} \sum_{k=0} ^{\infty} \frac{2^{\iota Q} 2^{-\Bar{N}(i+\iota)} \bigl( 2^{k+1} \varepsilon^{-1} \|(y,t)\| \bigr)^Q }  {\{ 2^{\iota} 2^k \varepsilon^{-1}\|(y,t)\|\}^{Q + \tau}}  \mathcal{M}f(x,u) \nonumber  \\
    & \leq  C \varepsilon^{\tau} \|(y,t)\|^{1-\tau} \mathcal{M}f(x,u) \left(\sum_{k=1} ^{\infty} \frac{1}{2^{k\tau}} \right) \left(\sum_{\iota \geq 0} \frac{1}{2^{\iota(\Bar{N}+\tau)}}  \right)\nonumber \\
    & \leq C \varepsilon^{\tau} \|(y,t)\|^{1-\tau} \mathcal{M}f(x,u) ,\nonumber
\end{align*}
where $\mathcal{M}$ denote the Hardy-Littlewood maximal operator defined on $G$ relative to the distance $\varrho$. 

On the other hand, using Proposition (\ref{Proposition: Weighted Plancherel with L infinity condition}), (\ref{equation: F^{i} with psi}) proceeding as above, we have 
\begin{align*}
 |J_{11}| \leq C \varepsilon^{\tau} 2^{i(Q +\tau + 1 + 2\epsilon)} \|(y,t)\|^{1-\tau} \mathcal{M}f(x,u).   
\end{align*}
Therefore combining both estimates $J_{11}$ and $J_{12}$, we get
\begin{align*}
    |J_1| &\leq C \varepsilon^{\tau} \|(y,t)\|^{1-\tau} (1 +  2^{i(Q +\tau + 1 + 2\epsilon)}) \mathcal{M}f(x,u) \leq C  \varepsilon^{\tau} \|(y,t)\|^{1-\tau} \mathcal{M}f(x,u) .
\end{align*}

Similarly as  in $J_1$,  we also break $J_2$ into two parts $J_{21}$ and $J_{22}$ corresponding to $\psi$ and $\sum_{\iota \geq 0} \phi_{\iota}$. Let $\gamma_0$ be a horizontal curve joining between two points $(x,u)$ and $((x,u)(y,t))$. Then for $J_{22}$, using mean value theorem combining with Proposition (\ref{Proposition: Weighted Plancherel with L infinity condition}) and (\ref{equation: decay of Fi with phi}), we get 
\begin{align*}
 & |J_{22}| \\
 &\leq C \sum_{\iota \geq 0} \int_{\|((x',u')^{-1}(x,u))\| >  \varepsilon^{-1} \|(y,t)\|} | b((x,u)(y,t))- b(x',u')| \\
 & \hspace{1.5cm} |\mathcal{K}_{(\varphi_{\iota} F^{(i)})(\sqrt{\mathcal{L}})} ((x',u')^{-1}(x,u))-  \mathcal{K}_{(\varphi_{\iota} F^{(i)})(\sqrt{\mathcal{L}})} ((x',u')^{-1}(x,u)(y,t))| f(x',u')\, d(x',u') \\
 & \leq C \|b\|_{L^{\infty}} \|(y,t)\| \sum_{\iota \geq 0} \int_{\|((x',u')^{-1}(x,u))\| >  \varepsilon^{-1} \|(y,t)\|} \int_{0}^{1} | X \mathcal{K}_{(\varphi_{\iota} F^{(i)})(\sqrt{\mathcal{L}})} ((x',u')^{-1}\gamma_0(s))| \\
 & \hspace{10cm} |f(x',u')|\, d(x',u') \ ds \\
 & \leq C \,  \|(y,t)\| \sum_{\iota \geq 0}  \sum_{k \geq 0} \int_{2^{k+1} \varepsilon^{-1} \|(y,t)\|\geq \|((x',u')^{-1}(x,u))\| > 2^k \varepsilon^{-1} \|(y,t)\|} \int_0^1 \\
 & \hspace{1.5cm} \frac{  (1 + 2^{\iota} \|((x',u')^{-1}\gamma_0(s))\|)^{Q+ \tau} | X \mathcal{K}_{(\varphi_{\iota} F^{(i)})(\sqrt{\mathcal{L}})} ((x',u')^{-1}\gamma_0(s))| }{(1 + 2^{\iota} \|((x',u')^{-1}\gamma_0(s))\|)^{Q+ \tau}}\, |f(x',u')|\, d(x',u') \ ds \\
 & \leq C \,  \|(y,t)\| \sum_{\iota \geq 0}  \sum_{k \geq 0}  \frac{2^{\iota (Q+1)} \|\phi_{\iota} F^{(i)} (2^{\iota} \cdot)\|_{L^{\infty}_{Q + \tau + \epsilon} } |B((x,u), 2^{k+1} \varepsilon^{-1} \|(y,t)\| ) | } { (  2^{\iota} 2^k \varepsilon^{-1} \|(y,t)\| )^{Q + \tau}}\\
 & \hspace{3cm}   \frac{1}{ |B((x,u), 2^{k+1} \varepsilon^{-1} \|(y,t)\| ) | } \int_{\varrho((x',u'),(x,u)) < 2^{k+1} \varepsilon^{-1} \|(y,t)\|  } |f(x',u')|\, d(x',u') \\
 & \leq C \,  \|(y,t)\|^{1-\tau} \, \varepsilon^{\tau} \mathcal{M}f(x,u),
\end{align*}
provided $\|(y,t)\|$ is very small and $\Bar{N}>1$. Here in the second last inequality we have used the fact $\|((x',u')^{-1}\gamma_0(s))\| \sim \|((x',u')^{-1}(x,u))\|$ if $\|(y,t)\|$ is very small. (Actually later we will estimate for $\|(y,t)\| \leq r$ and $r \to 0$.)

In a similar fashion, we also get
\begin{align*}
    |J_{21}| \leq C \varepsilon^{\tau} 2^{i(Q +\tau + 1 + 2\epsilon)} \|(y,t)\|^{1-\tau} \mathcal{M}f(x,u) .
\end{align*}
Then combining both the estimates for $J_{21}$ and $J_{22}$ we get, 
\begin{align*}
    |J_2| &\leq C \varepsilon^{\tau} \|(y,t)\|^{1-\tau} (1 +  2^{i(Q +\tau + 1 + 2\epsilon)}) \mathcal{M}f(x,u) \leq C_i \varepsilon^{\tau} \|(y,t)\|^{1-\tau}  \mathcal{M}f(x,u).
\end{align*}
In case of $J_3$, similarly as above, we again write $J_3$ as a sum of $J_{31}$ and $J_{32}$. Using mean value theorem, Proposition \ref{Proposition: Weighted Plancherel with L infinity condition} and (\ref{equation: F^{i} with psi}) we see that 
\begin{align*}
  & |J_{31}| \leq C \int_{\|((x',u')^{-1}(x,u))\| \leq  \varepsilon^{-1} \|(y,t)\|} \|((x',u')^{-1}(x,u))\| |K_{(\psi F^{(i)})(\sqrt{\mathcal{L}})} ((x',u')^{-1}(x,u))| \\ 
  & \hspace{9cm} \,|f(x',u')|\, d(x',u') \nonumber\\
  &\leq C \varepsilon^{-1} \|(y,t)\| \|\psi F^{(i)}\|_{L^{\infty}_{1}}\, |B((x,u),\varepsilon^{-1} \|(y,t)\|)|\nonumber \\
  & \hspace{5cm} \frac{1}{|B((x,u),\varepsilon^{-1} \|(y,t)\| )|} \int_{  B((x,u) ,\varepsilon^{-1} \|(y,t)\| )}   \,|f(x',u')|\, d(x',u') \nonumber\\
  &  \leq C 2^{i (3/2+\epsilon)} \varepsilon^{-1} \|(y,t)\| \, (\varepsilon^{-1} \|(y,t)\|)^Q\, \mathcal{M}f(x,u)\nonumber\\
  &  \leq C 2^{i (3/2+\epsilon)} (\varepsilon^{-1} \|(y,t)\|)^{Q+1}   \mathcal{M}f(x,u)\nonumber.
\end{align*}

For $J_{32}$, proceeding similarly as in $J_{31}$, we obtain that $|J_{32}| \leq C (\varepsilon^{-1} \|(y,t)\|)^{Q+1} \, \mathcal{M}f(x,u)$. Thus, combining both $J_{31}$ and $J_{32}$ we get
\begin{align*}
    |J_3| &\leq C(1+ 2^{i(1+2\epsilon)}) (\varepsilon^{-1} \|(y,t)\|)^{Q+1} \, \mathcal{M}f(x,u)  \leq C_i  (\varepsilon^{-1} \|(y,t)\|)^{Q+1} \, \mathcal{M}f(x,u).
\end{align*}

For $J_4$, again we decompose as in $J_1$ into two parts. Using Proposition \ref{Proposition: Weighted Plancherel with L infinity condition}, (\ref{equation: F^{i} with psi}), mean value theorem, and left invariance of $\varrho$, we see that
\begin{align*}
  & |J_{41}| \leq \int_{\|((x',u')^{-1}(x,u))\| \leq  \varepsilon^{-1} \|(y,t)\|} |\mathcal{K}_{(\psi F^{(i)})(\sqrt{\mathcal{L}})} ((x',u')^{-1}(x,u)(y,t))| \\ 
  & \hspace{7cm} \times |b((x,u)(y,t))- b(x',u')| |f(x',u')|\, d(x',u') \nonumber\\
  &\nonumber \leq C \int_{\varrho((x',u'), (x,u)) \leq  \varepsilon^{-1} \|(y,t)\|} 
  |\mathcal{K}_{(\psi F^{(i)})(\sqrt{\mathcal{L}})} ((x',u')^{-1}(x,u)(y,t))| \\
  & \hspace{7cm} \times \varrho((x,u)(y,t), (x',u'))  \, |f(x',u')|\, d(x',u') \nonumber\\
  & \leq C (\varepsilon^{-1} \|(y,t)\| + \varrho((x,u)(y,t),(x,u)) \|\psi F^{(i)}\|_{L^{\infty}_{1}} \int_{\varrho((x',u'),(x,u)) \leq  \varepsilon^{-1} \|(y,t)\|}\, |f(x',u')|\, d(x',u')  \nonumber\\ 
   & \leq C (\varepsilon^{-1} \|(y,t)\|)^{Q+1} 2^{i (3/2+\epsilon)}\, \mathcal{M}f(x,u) .
\end{align*}
A similar calculation also shows that $|J_{42}| \leq (\varepsilon^{-1} \|(y,t)\|)^{Q+1}\, \mathcal{M}f(x,u)$. Thus, combining both $J_{41}$ and $J_{42}$, we get
\begin{align*}
     |J_{4}| &\leq C_i  (\varepsilon^{-1} \|(y,t)\|)^{Q+1}\, \mathcal{M}f(x,u) .
\end{align*}
Now, gathering all the  estimates of $J_1, J_2, J_3$ and $J_4$, and using the $L^p$ ($p>1$) boundedness of $\mathcal{M}$, we see that  
\begin{align}
    & \int_G \left|  [b, F^{(i)} (\sqrt{\mathcal{L}})]f(x,u) - ([b, F^{(i)} (\sqrt{\mathcal{L}})]f)_{B((x,u,) r)} \right|^p\, d(x,u) \nonumber\\
    & \leq \int_G  \frac{1}{|B(0, r)|^p}  \Bigl( \int_{B(0, r)} \bigl( |J_1| +  |J_2| + |J_3|+ |J_4| \bigr) \, d(y,t)\, \Bigr)^p\, d(x,u) \nonumber\\
   & \leq C_i  \Bigl( \varepsilon ^\tau r^{1-\tau} + (\varepsilon^{-1} r)^{Q+1} \,  \Bigr)^p\, \int_G \mathcal{M}f (x,u) ^p d(x,u) \nonumber\\
   & \leq C  \Bigl( \varepsilon ^\tau r^{1-\tau} + (\varepsilon^{-1} r)^{Q+1} \,  \Bigr)^p \int_G |f (x,u)| ^p d(x,u) .\nonumber
\end{align}
Finally, taking $r < \varepsilon^2$, as $r \to 0$ we conclude that
\begin{align}
    & \int_G \left|  [b, F^{(i)} (\sqrt{\mathcal{L}})]f(x,u) - ([b, F^{(i)} (\sqrt{\mathcal{L}})]f)_{B((x,u), r)} \right|^p\, d(x,u) < \epsilon.\nonumber
\end{align}
This shows that $ [b, F^{(i)} (\sqrt{\mathcal{L}})]$ satisfies condition $(3)$ of Proposition (\ref{Proposition: Kolmogorov-Riesz compactness theorem}). This completes the proof of Theorem \ref{Theorem: Compactness of Bochner-Riesz commutator}.
\end{proof}

\section{Data availability}
No data was used for the research described in this article.

\section*{Acknowledgments}
The first author was supported from his NBHM post-doctoral fellowship, DAE, Government of India. The second author would like to acknowledge the support of the Prime Minister's Research Fellows (PMRF) supported by Ministry of Education, Government of India. We would like to thank Sayan Bagchi for his careful reading of the manuscript and numerous useful suggestions as well as discussions.

\section{Conflict of Interest Statement}
The authors declare no conflict of interest.


\newcommand{\etalchar}[1]{$^{#1}$}
\providecommand{\bysame}{\leavevmode\hbox to3em{\hrulefill}\thinspace}
\providecommand{\MR}{\relax\ifhmode\unskip\space\fi MR }
\providecommand{\MRhref}[2]{%
  \href{http://www.ams.org/mathscinet-getitem?mr=#1}{#2}
}
\providecommand{\href}[2]{#2}

\end{document}